\documentclass[11pt]{article}
\usepackage{latexsym,amsmath,amsfonts,hyper,setspace}
\def\nbw{section}

\numberwithin{equation}{\nbw}

\newtheorem{theorem}{Theorem}
\numberwithin{theorem}{\nbw}
\newtheorem{proposition}{Proposition}
\numberwithin{proposition}{\nbw}
\newtheorem{corollary}{Corollary}
\numberwithin{corollary}{\nbw}
\newtheorem{lemma}{Lemma}
\numberwithin{lemma}{\nbw}
\newtheorem{question}{Question}
\numberwithin{question}{\nbw}

\numberwithin{conjecture}{\nbw}

\newtheorem{assumption}{Assumption}
\numberwithin{assumption}{\nbw}

\numberwithin{definition}{\nbw}

\numberwithin{notation}{\nbw}

\numberwithin{condition}{\nbw}

\numberwithin{example}{\nbw}

\newtheorem{claim}{Claim}
\numberwithin{claim}{\nbw}

\newtheorem{remark}{Remark}
\numberwithin{remark}{\nbw}

\numberwithin{question}{\nbw}

\numberwithin{goal}{\nbw}

\numberwithin{fact}{\nbw}

\newcommand{\thmref}[1]{Theorem~\ref{thm:#1}} 
\newcommand{\lemref}[1]{Lemma~\ref{lem:#1}} 
\newcommand{\propref}[1]{Proposition~\ref{prop:#1}} 
\newcommand{\remref}[1]{Remark~\ref{rem:#1}} 
\newcommand{\corref}[1]{Corollary~\ref{cor:#1}} 
\newcommand{\assumref}[1]{Assumption~\ref{assum:#1}} 
\newcommand{\secref}[1]{Section~\ref{sec:#1}} 
\newcommand{\eqnref}[1]{(\ref{eq:#1})} 

\def\be{\begin{equation} }
\def\ee{ \end{equation}}

\def\ben{\begin{equation*}}
\def\een{\end{equation*}}
\def\bea{\begin{eqnarray}}
\def\eea{\end{eqnarray}}
\def\ee{\end{eqnarray}}
\def\bean{\begin{eqnarray*}}
\def\eean{\end{eqnarray*}}

\newcommand\ignore[1]{}

\def\R{\mathbb{R}} 
\def\C{\mathbb{C}} 
\def\N{\mathbb{N}} 

\newcommand{\Ex}[1]{\mathbb{E}\left[#1\right]} 
\newcommand{\Prwo}{\mathbb{P}} 

\newcommand{\Ind}[1]{\chi_{#1}} 
\newcommand{\Var}[1]{\mathbb{V}\left(#1\right)} 
\renewcommand{\Pr}[1]{\mathbb{P}\left(#1\right)} 

\newcommand{\bigoh}[1]{O\left(#1\right)}
\newcommand{\liloh}[1]{o\left(#1\right)}
\newcommand{\ohmega}[1]{\Omega\left(#1\right)}
\newcommand{\theita}[1]{\Theta\left(#1\right)}

\def\sC{\mathcal{C}}
\def\sE{\mathcal{E}}\def\sF{\mathcal{F}}
\def\sG{\mathcal{G}}\def\sH{\mathcal{H}}
\def\sL{\mathcal{L}}
\def\sM{\mathcal{M}}

\def\fG{{\sf G}}

\def\deg{{\rm deg}}

\newcommand\QED{\ifhmode\allowbreak\else\nobreak\fi
\quad\nobreak$\Box$\medbreak}
\newcommand{\proofstart}{\par\noindent\sl Proof:\rm\enspace}
\newcommand{\proofend}{\QED\par}
\newenvironment{proof}{\proofstart}{\proofend}

\def\eps{\epsilon}

\def\deg{\indeg}

\renewcommand{\deg}{\mbox{d}}

\def\tr{{\rm Tr}}

\def\avgA{A_{\bf p}^{{\rm typ}}}
\def\typER{A^{\rm typ}_{m}}
\def\typA{A_{\bf p}^{{\rm typ}}}
\def\typG{\fG_{\bf p}^{{\rm typ}}}
\def\typL{\sL_{\bf p}^{{\rm typ}}}
\def\Ap{A_{\bf p}}
\def\Gp{\fG_{\bf p}}

\def\sLp{\sL_{\bf p}}

\def\Id{I}

\def\cb{{\bf e}}

\begin{document}

\title{Concentration of the adjacency matrix and of the Laplacian in random graphs with independent edges}
\author{Roberto Imbuzeiro Oliveira\thanks{\texttt{rob.oliv@gmail.com} IMPA, Rio de Janeiro, RJ, Brazil, 22430-040. Work supported by a {\em Bolsa de Produtividade em Pesquisa} and a {\em Projeto Universal} from CNPq, Brazil.}} \maketitle
\begin{abstract}Consider any random graph model where potential edges appear independently, with possibly different probabilities, and assume that the minimum expected degree is $\omega(\ln n)$. We prove that the adjacency matrix and the Laplacian of that random graph are concentrated around the corresponding matrices of the weighted graph whose edge weights are the probabilities in the random model.

We apply this result to two different settings. In bond percolation,
we show that, whenever the minimum expected degree in the random
model is not too small, the Laplacian of the percolated graph is
typically close to that of the original graph. As a corollary, we
improve upon a bound for the spectral gap of the percolated graph
due to Chung and Horn.

We also consider {\em inhomogeneous random graphs} with average
degree $\gg \ln n$. In this case we show that the adjacency matrix
of the random graph can be approximated (in a suitable sense) by an
integral operator defined in terms of the attachment kernel
$\kappa$.

Our main proof tool might be of independent interest: a new
concentration inequality for {\em matrix martingales} that
generalizes Freedman's inequality for the standard scalar
setting.\end{abstract}

\section{Introduction}

Much of probabilistic combinatorics deals with questions of the
following type:

\begin{question}\label{question:main}Given a probability distribution over ``large" combinatorial
objects $X$ and a real-valued parameter $P=P(X)$ defined over such
objects, does there exists a typical value $P^{\rm typ}$ such that
$P(X)$ is very likely to be close to $P^{\rm typ}$?\end{question}

Starting with the seminal work of Shamir and Spencer
\cite{ShamirSpencer_ChromaticNumber} on the chromatic number of
$G_{n,p}$, many answers to instances of the above question have been
obtained via {\em concentration inequalities}, and developments in
the two fields have often gone hand in hand; see
\cite{Ledoux_ConcentrationBook,AlonSpencer_ProbMethod} and the
references therein for many examples.

In this paper we introduce a new concentration inequality for {\em
random Hermitian matrices} in order to address a variant of Question
\ref{question:main}. Our combinatorial objects consist of {\em
random graphs with independent edges}. These are random graphs where
the events ``$ij$ is an edge" (with $ij$ varying over all unordered
pairs of vertices) are independent, but not necessarily identically
distributed. The new twist is that the ``parameters" for which we
prove concentration are the {\em adjacency matrix} and the {\em
graph Laplacian} of the resulting graph (defined in
\secref{prelimgraph}).

We briefly recall why these two matrices are important. Many
(real-valued) parameters of a graph can be computed and/or estimated
from these two matrices, including the diameter, distances between
distinct subsets, discrepancy-like properties, path congestion,
chromatic number and the mixing time for random walk; see e.g.
\cite{Chung_SpectralGraphTheory} for a compendium of these results,
\cite{KrivelevichSudakov_Pseudorandom,ChungGraham_QuasirandomGraphs,ChungGraham_QuasirandomGraphsGivenDegrees,ChungGrahamWilson_QuasirandomGraphs}
for the relationship between the two matrices and ``pseudo-random"
properties of graphs and
\cite{AlonKahale_3Coloring,FeigeOfek_Spectral,CojaLanka_PlantedPartition}
for algorithmic applications. Given these facts, our main Theorem
(stated below) sheds some light on the typical properties of the
corresponding random graph models.

\begin{theorem}[Loosely stated]\label{thm:mainintro} Let $\Gp$ be a random graph on vertex set $[n]$ where each potential edge $ij$, $1\leq i\leq j\leq n$ appears with probability ${\bf p}(i,j)$. Let $\Ap$ and $\sLp$ be the adjacency
matrix and graph Laplacian of $\Gp$ and $\typA$ and $\typL$ be the
adjacency matrix and Laplacian of the weighted graph $\typG$ where
$ij$ has weight ${\bf p}(i,j)$ for each pair $ij$. Define $d$,
$\Delta$ as the minimum and maximal weighted degrees in $\typG$.
Then there exists a universal constant $C>0$ such that if
$\Delta\geq C\ln n$,
$$\|\Ap - \typA\| = \bigoh{\sqrt{\Delta\ln n}}\mbox{ with high probability}$$
and, if $d\geq C\ln n$,
$$\|\sLp - \typL\| = \bigoh{\sqrt{\frac{\ln n}{d}}}\mbox{ with high
probability.}$$\end{theorem}

A more precise quantitative statement of \thmref{mainintro} is given
in \secref{typicalmatrices} below.

\thmref{mainintro} is related to several known results about the
standard Erd\"{o}s-R\'{e}nyi graph $G_{n,p}$ (the special case where
${\bf p}(i,j)=p$ for $i\neq j$). We will show in
\secref{quasirandom} that the kind of matrix concentration we prove
here is implicit in the literature and that the standard notion of
{\em quasi-randomness} for dense graphs
\cite{ChungGraham_QuasirandomGraphs,KrivelevichSudakov_Pseudorandom}
can be reformulated in terms of concentration of the adjacency
matrix around the ``typical matrix" for the corresponding $G_{n,p}$
model. There is also a relationship between concentration of the
Laplacian and quasi-randomness for given degree sequences
\cite{ChungGraham_QuasirandomGraphsGivenDegrees,ChungGraham_QuasirandomGraphsGivenDegrees,ChungGraham_QuasirandomGraphs}
which is briefly discussed in \secref{concentrationerdosrenyi}.

For the special cases just described, the bounds obtained from
\thmref{mainintro} for the Laplacian are qualitatively sharp, in the
sense that they becomes trivial at roughly the same point where one
cannot expect concentration to hold. However, more specialized (and
much more complex) approaches yield improved bounds
\cite{FeigeOfek_Spectral,CojaLanka_SpectrumGivenDegrees,FurediKomlos_Eigenvalues}.
In some sense, this is due to the fact that the typical adjacency
matrices and Laplacians for such random graph models turn out to be
very degenerate: one of the eigenvalues of each matrix has
multiplicity $n-1$, and the other eigenvalue is well separated from
the first.

The cases where this does {\em not} happen turn out to be more
interesting. For instance, consider the case of {\em bond
percolation} with a parameter $p\in(0,1)$ on an arbitrary $n$-vertex
graph $G$. That is, we consider a random subgraph $G_p$ of $G$ that
is obtained by retaining each edge of $G$ independently with
probability $p$. Let $A$ be the adjacency matrix and $\sL$ be the
Laplacian of $G$ (respectively). We will show that when the minimum
expected degree in $G_p$ is $\omega(\ln n)$, the adjacency matrix
and Laplacian of $G_p$ are close to $pA$ and $\sL$ (respectively);
therefore, any estimate for $G$ derived from $\sL$ continues to hold
(at least approximately) for the random subgraph. A simple corollary
of our Theorem is a bound for the spectral gap of $G_p$ that
improves upon a recent result of Chung and Horn
\cite{ChungHorn_PercolationSpectrum}, derived via much more
complicated methods.

We then turn to the general model of {\em inhomogeneous random
graphs}. These are built from a set of points $X_1,\dots,X_n$ that
are uniformly distributed over $[0,1]$. The probability that $i$ and
$j$ are connected in the random graph is $p\kappa(X_i,X_j)$, where
$\kappa:M\times M\to\R_+$ is a symmetric function (called a {\em
kernel}) and $p$ is a parameter that controls the density of the
resulting graph. Under some technical conditions, we will show that,
for $p=\omega(\ln n/n)$, the adjacency matrix of the random graph
will correspond to a kind of discretization of an integral operator
$T_\kappa$ defined in terms of $\kappa$. \thmref{mainintro} takes
care of the key step where we show that the adjacency matrix is
concentrated around a deterministic matrix; the rest of the argument
consists of proving that the latter matrix is an approximation of
$T_\kappa$ in some suitable sense. The end result implies that the
random graph and the kernel $\kappa$ are close in a metric that is stronger than the {\em cut metric} from the literature on graph limits
\cite{BorgsEtAl_ConvergentSequencesDenseGraphs,BorgsEtAl_ConvergentSequencesDenseGraphsII,LovaszSzegedy_GraphLimits,BollobasRiordan_MetricsSparseGraphs}.
Our results also imply that the eigenvalue distributions and the
eigenvectors of the adjacency matrix of the random graph model are
closely related to those of $T_\kappa$.

\subsection{A new concentration inequality}

The main result, \thmref{typicalmatrices}, is a straightforward
consequence of a new concentration inequality for {\em random
matrices}. Our result bounds the fluctuations $Z-\Ex{Z}$ of certain
random $d\times d$ Hermitian matrices $Z$ from their mean (defined
entrywise), as measured by {\em largest eigenvalue}
$\lambda_{\max}(Z-\Ex{Z})$ and the {\em spectral norm}
$\|Z-\Ex{Z}\|$.

Not much is known in general about such inequalities. This is in
sharp contrast with the scalar case, where there are several
remarkable inequalities and many techniques to prove them
\cite{Ledoux_ConcentrationBook,BoucheronEtAl_MomentBounds,AlonSpencer_ProbMethod}.
The concentration results for random matrices that have been proven
correspond to relatively old developments in the scalar case, such
as the standard bounds due to Chernoff
\cite{Chernoff_Bound,AhlswedeWinter_StrongConverse} and Hoeffding
\cite{Hoeffding_Bound,ChristofidesMarkstrom_HoeffdingForMatrices},
as well as Khintchine's inequality
\cite{LustPicardPisier_Khintchine,Rudelson_RandomIsotropic}.
Accordingly, the new concentration result we introduce in this paper
is a matrix analogue of {\em Freedman's inequality} for martingale
sequences \cite{Freedman_Ineq}, which dates back to the 1970's. Here
is a precise statement. [Measurability and conditional expectations
are defined entrywise; see \secref{prelimprob} for this and other
definitions.]

\begin{theorem}[Freedman's Inequality for Matrix Martingales]\label{thm:freedman}Let $$0=Z_0,Z_1,\dots,Z_n$$
be a sequence of random $d\times d$ Hermitian matrices that forms a
martingale sequence with respect to the filtration
$\sF_0,\sF_1,\dots,\sF_n$ (that is, for each $1\leq i\leq n$ $Z_i$
is $\sF_i$-measurable and $\Ex{Z_{i}\mid\sF_{i-1}} = Z_{i-1}$).
Suppose further than $\|Z_i-Z_{i-1}\|\leq M$ almost surely for each
$1\leq i\leq n$ and define:
$$W_n\equiv \sum_{i=1}^{n} \Ex{(Z_i-Z_{i-1})^2\mid\sF_{i-1}}.$$
Then for all $t,\sigma>0$:
$$\Pr{\lambda_{\max}(Z_n)\geq t,\, \lambda_{\max}(W_n)\leq \sigma^2} \leq d\,e^{-\frac{t^2}{8\sigma^2 +
4Mt}}.$$\end{theorem}

Compared with Freedman's original bound, \thmref{freedman} has worse
constants in the exponent and an extra $d$ factor (which is
necessary; cf. \secref{final}), but the two bounds are otherwise of
the same form. In this paper we only need a version of
\thmref{freedman} for independent sums (cf. \remref{indepcase} and
\corref{freedman}), but the martingale inequality is not any harder
to prove.

The proof of \thmref{freedman} follows a methodology first proposed
 by Ahlswede and Winter \cite{AhlswedeWinter_StrongConverse}. These authors proved a version of the Chernoff bound for matrices which has had a very strong impact on the development of Quantum
Information Theory
\cite{Devetak_Wiretap,DevetakWinter_Trilogy,Winter_ConvexDecompPOVM}. Christofides and Markstr\"{o}m
\cite{ChristofidesMarkstrom_HoeffdingForMatrices} used the
same method to obtain a version of Hoeffding's inequality
for matrix martingales.

\begin{theorem}[\cite{ChristofidesMarkstrom_HoeffdingForMatrices}, in abridged form]\label{thm:hoeffding}In the setting of \thmref{freedman}, replace the assumption on $\|Z_i-Z_{i-1}\|$ by the assumption that there exist $0\leq r_i\leq 1$ such that $\lambda_{\max}(Z_i-Z_{i-1})\leq 1-r_i$ and
$\lambda_{\max}(Z_{i-1}-Z_{i})\leq r_i$. Then for all $t>0$,
$$\Pr{\lambda_{\max}(Z_n)\geq t} \leq d\,e^{-nH_{R/n}\left(\frac{R+t}{n}\right)}$$
where $R = \sum_{i=1}^n r_i/n$ and for $x,r\in[0,1]$
$$H_r(x) \equiv x \ln\left(\frac{x}{p}\right) + (1 - x)
\ln\left(\frac{1-x}{1-p}\right).$$\end{theorem}

As we will see in \remref{comparisons}, this bound would not suffice
for our applications. Roughly speaking, our Theorem is better
because the variance term in our bound is the largest value of a sum
of matrices, not the sum of the largest eigenvalues. In that
respect, \thmref{freedman} is closer to an influential bound
obtained by Rudelson \cite{Rudelson_RandomIsotropic} via certain
inequalities from non-commutative probability
\cite{LustPicardPisier_Khintchine}. The Ahlswede-Winter approach we
adopt here has the advantage of requiring no such unfamiliar
tools.\footnote{There is now a proof of Rudelson's bound along the
lines of the Ahlswede-Winter method; see \cite{Eu_Rudelson} for
details and for further discussion on the difference between the
three bounds.}

\thmref{freedman} should also be contrasted with other ways for
controlling eigenvalues and eigenvectors of random matrices. One of
them is the ``trace method"
\cite{ChungHorn_PercolationSpectrum,ChungLuVu_SpectraGivenDegrees,FurediKomlos_Eigenvalues,Friedman_RelativelyRamanujan,Friedman_SecondEigenvalue}
which consists of analyzing traces of high powers of the matrices
under consideration. This method can be very sharp, but it is also
quite complex and we will see that we obtain better bounds in one
context (but not all contexts) where the trace method has been
applied. A more recent way of bounding eigenvalues and eigenvectors
is in some sense based on bounding ``discrepancies"
\cite{FeigeOfek_Spectral,CojaLanka_SpectrumGivenDegrees}. This is
better than our bound when the technique applies (see e.g. the
comments in \secref{concentrationerdosrenyi}), but our main
applications seem to be beyond the reach of this methodology.

Finally, we note that our result is not quite comparable
concentration bounds of Alon, Krivelevich and Vu
\cite{AlonKrivelevichVu_Concentration} for the largest eigenvalues
of a random symmetric matrix. Our bound is poorer than theirs when
applied to $k$-th largest eigenvalue for any fixed $k$, but their
bound quickly deteriorates when $k$ grows, whereas our result bounds
the maximal deviation of all eigenvalues simultaneously (cf.
\corref{typicaleigenvalues}), as well as the deviation of
eigenspaces (cf. \corref{typicaleigenvectors}). Moreover, their
result cannot be used to determine the typical value of each
eigenvalue.

\subsection{Organization}

The remainder of the paper is organized as follows. After the
preliminary \secref{firstprelim}, we prove the main concentration
result in \secref{typicalmatrices}. As a test case, we apply our
results to the Erd\"{o}s-R\'{e}nyi random graph \secref{quasirandom}
where the connection with quasi-randomness is also discussed. Bond
percolation is discussed in \secref{bondpercolation}. The more
complicated case of inhomogeneous random graphs is treated in
\secref{inhomogeneous}, where we also compare our results to what is
known about graph limits. The new concentration inequality is proven
in \secref{freedman}. Some final remarks are made in \secref{final}.
The Appendix contains two simple results on the perturbation theory
of compact operators for which we did not find adequate references.

\section{Preliminaries}\label{sec:firstprelim}

\subsection{Matrix notation}

The space of $d_r\times d_c$ matrices with real (resp. complex)
entries will be denoted by $\R^{d_r\times d_c}$ (resp.
$\C^{d_r\times d_c}$). Moreover, for $A\in \C^{d_r\times d_c}$,
$A^*\in\C^{d_c\times d_r}$ is the conjugate transpose of $A$. We
will identify $\R^d$ (resp. $\C^d$) with the space $\R^{d\times 1}$
(resp. $\C^{d\times 1}$) of column vectors, so that the inner
product of $v,w\in \R^d$ is $w^*v$. $\|\cdot\|$ denotes both the
Eucliean norm on $\R^d$ or $\R^d$ and the {\em spectral radius} norm
induced on $\C^{d\times d'}$:
$$\|A\|\equiv \sup_{v\in\C^d,\|v\|=1}\|Av\|, A\in\C^{d\times
d'}.$$

$\C^{d\times d}_{\rm Herm}$ is the space of $d\times d$ {\em
Hermitian matrices}, which are the $A\in\C^{d\times d}$ with
$A^*=A$. $\R^{d\times d}_{\rm Herm}$ is similarly defined; one could
of course speak of {\em symmetric} matrices in this case and use
$A^\dag$ instead of $A^*$, but we will keep notation consistent.

The spectral theorem implies that for any $A\in \C_{\rm
Herm}^{d\times d}$ there exist real numbers $\lambda_0(A)\leq
\lambda_1(A)\leq \dots\leq \lambda_{d-1}(A)$ and orthonormal vectors
$\psi_0,\dots,\psi_{d-1}$ (the eigenvalues and eigenvectors of $A$,
respectively) with:
$$A\equiv \sum_{i=0}^{d-1}\lambda_i(A)\,\psi_i\psi_i^*.$$
The {\em spectrum of $A$} is the set ${\rm spec}(A)$ of all
$\lambda_i(A)$. The above formula implies that for $A\in \C^{d\times
d}_{\rm Herm}$:
$$\|A\| = \max_{0\leq i\leq d-1}|\lambda_i(A)| =
\max_{v\in\C^d,\|v\|=1}v^*Av.$$ For $A\in\R^{d\times d}_{\rm Herm}$,
the eigenvectors of $A$ are all real and one only needs to maximize
over $v\in\R^d$ in the above formula to compute $\|A\|$.

We also note an equivalent statement of the spectral theorem as:
$$A = \sum_{\alpha\in {\rm spec}(A)}\,\alpha\,\Pi_\alpha,$$
where the $\{\Pi_\alpha\}_{\alpha\in {\rm spec}(A)}$ are projections
with orthogonal ranges and $\sum_{\alpha\in {\rm
spec}(A)}\Pi_\alpha=\Id_d$, the $d\times d$ identity matrix. The
{\em multiplicity} of $\alpha\in {\rm spec}(A)$ is the dimension of
the range of the corresponding $\Pi_\alpha$; this is equal to the
number of $0\leq i\leq d-1$ with $\lambda_i(A)=\alpha$.

\subsection{Integral operators on $L^2([0,1])$ and spectral theory}\label{sec:integralop}

In \secref{inhomogeneous} we will compare adjacency matrices with
certain integral operators on $L^2([0,1])$. The spectral theory of
these and other compact operators is a classical topic in Functional Analysis and we
refer to
\cite{RieszSzNagy_FunctionalAnalysis,Rudin_FunctionalAnalysis} for
all the results we review in this Section.

We will work with the space $L^2([0,1])$ of real measurable
functions that are square-integrable with respect to Lebesgue
measure. This space has a natural inner product
$$(f,g)_{L^2}\equiv \int_0^1\,f(x)\,g(x)\,dx\;\;(f,g\in L^2([0,1]))$$
and an associated norm $\|f\|_{L^2}^2\equiv (f,f)_{L^2}$ with
respect to which it is a real Hilbert space.

Given a function $\eta\in L^2([0,1]^2)$ (the latter space being
defined similarly to $L^2([0,1])$), one can define a linear operator
on $L^2([0,1])$ by the formula:
\begin{equation}\label{eq:integralop1}T_\eta:f(\cdot)\in
L^2([0,1])\mapsto (T_\eta f)(\cdot) \equiv
\int_0^1\eta(\cdot,y)\,f(y)\,dy.\end{equation}

The ``$L^2\to L^2$" norm of a linear operator $V$ from $L^2([0,1])$
to itself is given by:
$$\|V\|_{L^2\to L^2} \equiv \sup_{f\in
L^2([0,1])\backslash\{0\}}\frac{\|Vf\|_{L^2}}{\|f\|_{L^2}}.$$ It is
an exercise to show via the Cauchy Schwartz inequality that:
\begin{equation}\label{eq:normofTeta}\|T_\eta\|_{L^2\to L^2}^2\leq
\int_{[0,1]^2}\eta^2(x,y)\,dx\,dy.\end{equation} Moreover, if
$\eta':[0,1]^2\to\R$ is also square integrable, $T_\eta-T_{\eta'}$
equals $T_{\eta-\eta'}$.

Assume that $\eta(x,y)=\eta(y,x)$ for almost every $(x,y)\in[0,1]$
(i.e. $\eta$ is symmetric). In that case the operator $T_\eta$ is a
{\em compact, self adjoint} linear operator on the Hilbert space $L^2([0,1])$.

Let us recall what these properties imply. Let $T$ be a bounded,
compact, self-adjoint operator on the Hilbert space $\sH$. Then
there exists a finite or countable set $S\subset \R$ and a family
$\{P_\alpha\,:\, \alpha\in S\}$ of orthogonal projection operators
on $\sH$ with orthogonal ranges such that:
$$T =\sum_{\alpha \in S}\alpha\,P_\alpha\mbox{ and }{\rm Id}_{\sH}=\mbox{ identity operator on $\sH$}=\sum_{\alpha\in S}P_\alpha.$$
Moreover, either $S$ is finite and contains $0$, or $S$ is a
countable, bounded subset of $\R$ with $0$ as its only accumulation
point. Finally, all $P_\alpha$ for $\alpha\neq 0$ are finitely
dimensional; the {\em multiplicity} of $\alpha$ is precisely the
dimension of the range of $P_\alpha$. The spectrum of $T$ is the set
${\rm spec}(T_\eta)=S\cup\{0\}$.

\subsection{Concepts from Graph Theory}\label{sec:prelimgraph}
 For our purposes a graph
$G=(V,E)$ consists of a finite set $V$ of vertices and a set $E$ of
edges, which are subsets of size $1$ (loops) or $2$ of $V$ (we do
not allow for parallel edges). Unless otherwise noted, we will
assume that $V=[n]$ for some integer $n\geq 2$, where $[n]\equiv
\{1,2,\dots,n\}$. We will write edges as pairs $ij$ (allowing for
$i=j$), but we make no distinction between $ij$ and $ji$. We will
also write $i\sim_G j$ to mean that $ij\in E$. The {\em degree}
$\deg_G(i)$ of a vertex $i$ is the number of $1\leq j\leq n$ such
that $ij\in E$.

Assume that $V=[n]$. The {\em adjacency matrix} of $G$ is the
$n\times n$ matrix $A=A_G$ such that, for all $1\leq i,j\leq n$, the
$(i,j)$-th entry of $A$ is $1$ if $ij\in E$ and $0$ otherwise. The
{\em Laplacian} $\sL=\sL_G$ of $G$ is the matrix:
$$\sL_G = \Id_n - T_G\,A_G\,T_G$$
where $T$ is the $n\times n$ diagonal matrix whose $(i,i)$-th entry
is $\deg_G(i)^{-1/2}$ if $\deg_G(i)\neq 0$, or $0$ if $\deg_G(i)=0$.
We also let
$$\lambda(G)\equiv
\min\{\lambda_1(\sL),2-\lambda_{d-1}(\sL)\}$$ denote the {\em
spectral gap} of $G$.

We will also consider {\em weighted graphs}, which correspond to a
graph $H=(V',E')$ where a positive weight $w_e>0$ is assigned to
each edge $e\in E$. This is the same as defining a symmetric
function $w:(V')^2\to [0,+\infty)$ (i.e. $w(i,j)=w(j,i)\geq 0$ for
all $i,j\in V$) and setting $E'=\{\{i,j\}\,:\, w(i,j)>0\}$. In this
case, the {\em degree} of $i\in V'$ is defined as
$$\deg_H(i)\equiv \sum_{j=1}^nw(i,j).$$

Assume $V'=[m]$. The adjacency matrix of such an $H$ is the $m\times
m$ matrix $A_H$ where for each $1\leq i,j\leq m$ the $(i,j)$-th
entry of $A_H$ is $w(i,j)$. The Laplacian $\sL_H$ is defined as
$$\sL_H\equiv \Id_m - T_H A_H T_H,$$
where $T_H$ is defined as before, but with the new notion of degree.
The definition of $\lambda(H)$ is the same as for unweighted graphs.

\subsection{Probability with matrices}\label{sec:prelimprob}

We will be dealing with random Hermitian matrices throughout the
paper. Following common practice, we will always assume that we have
a probability space $(\Omega,\sF,\Prwo)$ in the background where all
random variables are defined.

Call a map $X:\Omega\to\C^{d\times d}_{\rm Herm}$ a random $d\times
d$ Hermitian matrix (or a $\C^{d\times d}_{\rm Herm}$-valued random
variable) if for each $1\leq i,j\leq n$, the function
$X(i,j):\Omega\to \C^{d\times d}$ corresponding to the $(i,j)$-th
entry of $X$ is $\sF$-measurable. We say that $X$ is {\em
integrable} if all of these entries are and let $\Ex{X}$ be the
matrix whose $(i,j)$-th entry is $\Ex{X(i,j)}$. Conditional
expectations with respect to a sub $\sigma$-field are also defined
entrywise.

If the entries are also square-integrable, one can define the
variance by the usual formula,
$$\Var{X} = \Ex{(X-\Ex{X})^2}.$$The standard identity $\Var{X} = \Ex{X^2}-\Ex{X}^2$
also holds in this setting.

We will need two easily checked properties of matrix (conditional)
expectations, valid for all integrable random $d\times d$ Hermitian
matrices $X$ and $Y$ and any sub-$\sigma$-field $\sG\subset\sF$:
\begin{eqnarray}\ignore{\label{eq:indep}\mbox{{\bf [Product rule]} If $X$ and $Y$ are
independent, }\Ex{XY} &=& \Ex{X}\Ex{Y}.\\}
\label{eq:traceexp}\mbox{\bf [$\tr$ and $\Ex{\dots}$ commute] }\tr(\Ex{X}) &=& \Ex{\tr(X)}.\\
\nonumber \mbox{{\bf [Conditioning]} If $Y$ is
$\sG$-measurable, } \Ex{XY\mid \sG} &=& \Ex{X\mid\sG}Y\\
\label{eq:product}\mbox{and }\Ex{XY}&=& \Ex{\Ex{X\mid\sG}Y}.\end{eqnarray}

\section{Concentration of graph matrices}\label{sec:typicalmatrices}

In this section we state and prove our main result,
\thmref{mainintro}.

Given $n\in\N\backslash\{0,1\}$, let ${\bf p}:[n]^2\to [0,1]$ be
symmetric: ${\bf p}(i,j)={\bf p}(j,i)$ for all $1\leq i,j\leq n$.
Define independent $0/1$ random variables $\{I_{ij}\,:\, 1\leq i\leq
j\leq n\}$ with
$$\Pr{I_{ij}=1}=1-\Pr{I_{ij}=0} = {\bf p}(i,j).$$
We also define $I_{ji}=I_{ij}$ for $j>i$.

Define a random unweighted graph $\Gp$ with vertex set $[n]$ and
edge set
$$E_{\bf p} \equiv \{ij\,:\, 1\leq i\leq j\leq n,\, I_{ij}=1\}.$$

Let $\Ap$ and $\sLp$ be the adjacency matrix and Laplacian of the
graph $\Gp$. We will compare these to the corresponding matrices
$\typA$, $\typL$ of the weighted graph $\typG$ defined by the
function ${\bf p}$.

The following is a more precise statement of \thmref{mainintro}.

\begin{theorem}[Existence of typical graph matrices]\label{thm:typicalmatrices}For any constant $c>0$ there exists another constant $C=C(c)>0$, independent of $n$ or ${\bf p}$, such that the following holds. Let $d\equiv\min_{i\in[n]}\deg_{\typG}(i)$, $\Delta\equiv\max_{i\in[n]}\deg_{\typG}(i)$. If $\Delta>C\ln n$, then for all $n^{-c}\leq \delta\leq 1/2$,
$$\Pr{\|\Ap - \avgA\|\leq 4\,\sqrt{\Delta\,\ln (n/\delta)}}\geq 1-\delta.$$
Moreover, if $d\geq C\ln n$, then for the same range of $\delta$:
$$\Pr{\|\sLp - \typL\|\leq 14\,\sqrt{\frac{\ln(4n/\delta)}{d}}}\geq 1-\delta.$$\end{theorem}

We will quickly derive some corollaries before we prove
\thmref{typicalmatrices}.

Let $B_1,B_2\in\R^{n\times n}_{\rm Herm}$. Standard eigenvalue
interlacing inequalities \cite{HornJohnson_MatrixAnalysis} imply:
\begin{equation}\label{eq:eigenperturbation}\max_{i\in\{0,\dots,{n-1}\}}|\lambda_i(B_1)-\lambda_i(B_2)|\leq
\|B_1-B_2\|.\end{equation}

This immediately implies that:
\begin{corollary}\label{cor:typicaleigenvalues}In the setting of
\thmref{typicalmatrices},
$$\|\Ap - \avgA\|\leq 4\,\sqrt{\Delta\,\ln (n/\delta)}\Rightarrow \forall 0\leq i\leq n-1,\, |\lambda_i(\Ap)-\lambda_i(\typA)|\leq 4\,\sqrt{\Delta\,\ln (n/\delta)}.$$
Therefore, the RHS holds with probability $\geq 1-\delta$ for any
$n^{-c}<\delta<1/2$ if $\Delta\geq C\ln n$. Similarly,
$$\|\sLp - \typL\|\leq 14\,\sqrt{\frac{\ln(4n/\delta)}{d}}\Rightarrow \forall 0\leq i\leq n-1,\, |\lambda_i(\sLp) - \lambda_i(\typL)|\leq 14\,\sqrt{\frac{\ln(4n/\delta)}{d}},$$
and the RHS holds with probability $\geq 1-\delta$ for all $\delta$
as above if $d\geq C\ln n$.\end{corollary}

Now consider some $B\in\R^{n\times n}_{\rm Herm}$ and, for $a<b$
real, let $\Pi_{a,b}(B)$ be the orthogonal projector onto the space
spanned by the eigenvectors of $B$ corresponding to eigenvalues in
$[a,b]$. The following corollary is a consequence of
\lemref{perturbation} in the Appendix, as all operators on a
finite-dimensional Hilbert space are compact.

\begin{corollary}\label{cor:typicaleigenvectors}Given some $\gamma>0$, let $N_\gamma(\typA)$ be the set of all pairs $a<b$ such
$a+\gamma<b-\gamma$ and $\typA$ has no eigenvalues in
$(a-\gamma,a+\gamma)\cup (b-\gamma,b+\gamma)$. Then for
$\gamma>4\,\sqrt{\Delta\,\ln (n/\delta)}$, \begin{multline*}\|\Ap -
\avgA\|\leq 4\,\sqrt{\Delta\,\ln (n/\delta)}\\ \Rightarrow \forall
(a,b)\in N_\gamma(\typA),\, \|\Pi_{a,b}(\Ap)-\Pi_{a,b}(\typA)\|\leq
\left(\frac{4(b-a+2\gamma)}{\pi(\gamma^2-\gamma\sqrt{\Delta\,\ln
(n/\delta)})}\right)\,\sqrt{\Delta\,\ln (n/\delta)}.\end{multline*}
In particular, the RHS holds with probability $\geq 1-\delta$ for
any $n^{-c}<\delta<1/2$.

Define $N_\gamma(\typL)$ similarly. Then for
$\gamma>14\sqrt{{\ln(4n/\delta)/d}}$, \begin{multline*}\|\sLp -
\typL\|\leq 14\,\sqrt{\frac{\ln(4n/\delta)}{d}}\\ \Rightarrow
\forall (a,b)\in N_\gamma(\typA)\,
\|\Pi_{a,b}(\sLp)-\Pi_{a,b}(\typL)\|\leq
\left(\frac{14(b-a+2\gamma)}{\pi(\gamma^2-\gamma\sqrt{\frac{\ln
(4n/\delta)}{d}})}\right)\,\sqrt{\frac{\ln
(4n/\delta)}{d}}.\end{multline*} In particular, the RHS holds with
probability $\geq 1-\delta$ for any $n^{-c}<\delta<1/2$.
\end{corollary}

The upshot is that for any range of eigenvalues of $\typA$ (resp.
$\typL$) that are well-separated from the rest of the spectrum, the
projection onto the corresponding eigenvectors of $A$ (resp. $\sL$)
will be typically close to that of $\typA$ (resp.
$\typL$)\footnote{Of course, there is not much one can do near
eigenvalue degeneracies, where eigenvectors are typically
unstable.}. We will see when dealing with inhomogeneous random
graphs that the separation conditions demanded by the corollary are
satisfied in non-trivial cases.

\subsection{Proof of the concentration result}

We now prove \thmref{typicalmatrices}.

\begin{proof}~[of \thmref{typicalmatrices}]~\ignore{In this proof we use the notation $\R^{n\times n}_{\rm Herm}$ for the set of $n\times n$ Hermitian
matrices. $\R^n$ will be identified with the space of column vectors
$v$, whose transposes are denoted by $v^*$. In this way, the inner
product of $w$ and $v$ is $w^*v$. Moreover, for $B\in\R^{n\times
n}_{\rm Herm}$, $\|B\|$ will denote the {\em spectral norm} of $B$,
which is the maximum of the absolute values of its eigenvalues.
Finally, if $Z$ is a $n\times n$ Hermitian random matrix with
integrable entries $Z_{ij}$, its expectation $\Ex{Z}$ is the
$n\times n$ Hermitian matrix whose $(i,j)$-th entry is $\Ex{Z_{ij}}$
[This and other facts about matrices and probabilities with them are
all covered in \secref{prelim} below, to which we refer for other
issues.]

In this proof (as in the remainder of the paper) the {\em
expectation} $\Ex{V}$ of a random matrix $V$ is the matrix whose
entries are the expected values of the entries of $V$. The {\em
variance} of $V$ is $\Ex{(V-\Ex{V})^2}$, which is
$\Ex{V^2}-\Ex{V}^2$ as in the standard case. [This and other facts
about matrices and probabilities are covered in \secref{prelim}
below, to which we refer for other issues.]

}Let $\{\cb_i\}_{i=1}^n$ be the canonical basis for $\R^n$ . For each
$1\leq i,j\leq n$, define a corresponding matrix $A_{ij}$:
\begin{equation}\label{eq:defAe}A_{ij}\equiv \left\{\begin{array}{ll} \cb_i\cb_j^* + \cb_j\cb_i^*,&i\neq j;\\ \cb_i\cb_i^*, & i=j.\end{array}\right.\in \R^{n\times n}_{\rm
Herm}.\end{equation} One can check that $\Ap = \sum_{1\leq i\leq
j\leq n}I_{ij}\,A_{ij}$ and $\avgA = \sum_{1\leq i\leq j\leq n}{\bf
p}(i,j)A_{ij}$. Therefore,
\begin{equation*}\label{eq:adjacencyasmartingale}\Ap - \avgA = \sum_{1\leq i\leq j\leq n}X_{ij}\mbox{ where } X_{ij}\equiv (I_{ij} - {\bf
p}(i,j))\, A_{ij},\; 1\leq i\leq j\leq n.\end{equation*}

We wish to apply \thmref{freedman} (or rather, \corref{freedman} in
\secref{freedman}) to the above sum. To do this, we first notice
that the random matrices $X_{ij}$, which take values in $\C^{n\times
n}_{\rm Herm}$, are independent (since the $I_{ij}$ are) and have
mean zero (since $\Ex{I_{ij}}={\bf p}(i,j)$). Moreover,
$$\|X_{ij}\|\leq \|A_{ij}\| = 1$$ as the eigenvalues of $A_{ij}$
are always contained in the set $\{1,0,-1\}$ . Thus the assumptions
of the Corollary apply with $M=1$, but we still need to compute the
sum of the variances. For this, fix some pair $ij$ and note that:
$$\Ex{X_{ij}^2} = \Ex{(I_{ij}-{\bf p}(i,j))^2A^2_{ij}} = {\bf p}(i,j)(1-{\bf
p}(i,j))A^2_{ij}$$ and a computation reveals that
\begin{equation}\label{eq:Aijsquared}A^2_{ij}=\left\{\begin{array}{ll}\cb_i\cb_i^* + \cb_j\cb_j^*, &
i\neq j\\ \cb_i\cb_i^*, & i=j.\end{array}\right.\end{equation}
Therefore,
\begin{eqnarray*}\sum_{i\leq j}\Ex{X_{ij}^2} &=& \sum_i {\bf p}(i,i)(1-{\bf
p}(i,i))\cb_i\cb_i^* + \sum_{i<j} {\bf
p}(i,j)(1-{\bf p}(i,j))(\cb_i\cb_i^*+\cb_j\cb_j^*)\\
&=& \sum_{i=1}^n\left(\sum_{j=1}^n{\bf p}(i,j)(1-{\bf
p}(i,j))\right)\,\cb_i\cb_i^*.\end{eqnarray*} This is a diagonal
matrix and its largest eigenvalue is at most
$$\max_{i\in[n]}\left(\sum_{j=1}^n{\bf p}(i,j)(1-{\bf
p}(i,j))\right)\leq \max_{i\in[n]}\sum_{j=1}^n{\bf p}(i,j)=\Delta.$$
One can now apply \corref{freedman} with $\sigma^2=\Delta$ and $M=1$
to obtain:
\begin{equation}\label{eq:boundfinalpercolation}\forall
t>0,\,\Pr{\|\Ap-\avgA\|\geq t}\leq 2n\,e^{-\frac{t^2}{8\Delta +
4t}}.\end{equation}

Now let $c>0$ be given and assume $n^{-c}\leq \delta\leq 1/2$. Then
it is clear that there exists a $C=C(c)$ independent of $n$ and
${\bf p}$ such that whenever $\Delta\geq C\ln n$,
$$t=4\,\sqrt{\Delta \ln(2n/\delta)}\leq 2\Delta.$$
Plugging this $t$ into \eqnref{boundfinalpercolation} yields:
$$\Pr{\|\Ap-\avgA\|\geq 4\,\sqrt{\Delta \ln(2n/\delta)}}\leq 2n e^{-\frac{t^2}{16\Delta}} = 2n\,e^{-\frac{16\Delta\ln(2n/\delta)}{16\Delta}} = \delta.$$
This proves the first inequality in \thmref{typicalmatrices}.

In order to prove the second inequality, we again fix $n^{-c}\leq
\delta\leq 1/2$. Our first task is to control the vertex degrees in
$\Gp$. Notice that for each $1\leq i\leq n$,
$\deg_{\Gp}(i)=\sum_{j=1}^nI_{ij}$ is a sum of independent indicator
random variables and the mean of that sum is $\deg_{\typG}(i)\geq
d$. Standard Chernoff bounds \cite{AlonSpencer_ProbMethod} (or the
case $d=1$ of our own \corref{freedman}!) imply that there exists a
value of $C=C(c)$ such that for $d\geq C\ln n$,
$$\forall i\in[n],\, \Pr{\left|\frac{d_{\Gp}(i)}{d_{\typG}(i)}-1\right|>
4\sqrt{\frac{\ln(4n/\delta)}{d}}}\leq \delta/2n.$$

Thus with probability $\geq 1-\delta/2$ one has that
\begin{equation}\label{eq:simplebound}\forall i\in[n],
\left|\frac{d_{\Gp}(i)}{d_{\typG}(i)}-1\right|\leq
4\sqrt{\frac{\ln(4n/\delta)}{d}}.\end{equation}

We will use this inequality to compare the matrices $$T = \mbox{
diagonal with $\deg_{\Gp}(i)^{-1/2}$ at the $(i,i)$th position}$$
and
$$T_{\rm typ} = \mbox{ diagonal with $\deg_{\typG}(i)^{-1/2}$ at the $(i,i)$th position}$$

By increasing $C$ if necessary (and recalling that $\delta>n^{-c}$,
$d>C\ln n$), we can ensure that the RHS of \eqnref{simplebound} is
at most $3/4$. By the Mean Value Theorem for any $x\in [-3/4,3/4]$:
$$|\sqrt{1+x} - 1|\leq \left(\sup_{\theta\in
[-3/4,3/4]}\frac{1}{2\sqrt{1+\theta}}\right)\,|x| = \,x.$$ Applying
this to
$$x\equiv \frac{d_{\Gp}(i)}{d_{\typG}(i)}-1$$
yields that:
\begin{eqnarray}\nonumber \|TT_{\rm typ}^{-1}-I\| &=& \max_{1\leq i\leq
n}\left|\frac{\sqrt{\deg_{\Gp}(i)}}{\sqrt{\deg_{\typG}(i)}} - 1 \right|\\
\label{eq:TT_p}&\leq & 4\sqrt{\frac{\ln(4n/\delta)}{d}}\mbox{with
probability $\geq 1-\delta/2$.}\end{eqnarray}

We now wish compare $\sLp = I - T\Ap T$ to $\typL= I-T_{\rm
typ}\avgA T_{\rm typ}$. Introduce an intermediate operator:
\begin{equation}\sM\equiv I - T_{\rm typ}\Ap T_{\rm
typ}.\end{equation} A calculation reveals that:
$$\sM = I - (TT_{\rm typ}^{-1})(I-\sLp)(TT_{\rm typ}^{-1})$$
The spectrum of any Laplacian lies in $[0,2]$
\cite{Chung_SpectralGraphTheory}; this implies $\|I-\sLp\|\leq 1$.
Using this in conjunction with \eqnref{TT_p} yields:
\begin{eqnarray}\nonumber \|\sM -\sLp\| &=& \|(TT_{\rm typ}^{-1})(I-\sLp)(TT_{\rm typ}^{-1}) - (I - \sLp)\|\\
\nonumber &\leq &  \|(TT_{\rm typ}^{-1}-I)(I-\sLp)(TT_{\rm
typ}^{-1})\|
\\ \nonumber & & + \|(I-\sLp)(TT_{\rm typ}^{-1})\|\\ \mbox{(use ``$\|ABC\|\leq \|A\|\|B\|\|C\|$")}\nonumber
&\leq & \|TT_{\rm typ}^{-1}-I\|\,\|I-\sLp\|\, \|TT_{\rm typ}^{-1}\|
\\ \nonumber & & + \|I-\sLp\|\,\|TT_{\rm typ}^{-1}-I\|\\
\nonumber & \leq &4\sqrt{\frac{\ln(4n/\delta)}{d}}
\left(1+4\sqrt{\frac{\ln(4n/\delta)}{d}}\right)
 + 4\sqrt{\frac{\ln(4n/\delta)}{d}} \\ \nonumber & \leq & 10\sqrt{\frac{\ln(4n/\delta)}{d}}\mbox{ with probability }\geq 1-\delta/2,\end{eqnarray}
 where again we increase $C$ if necessary to ensure that $d\geq C\ln
 n$ and $\delta>n^{-c}$ imply the desired bound.

 To finish the proof, we must show that $\|\sM - \typL\|
\leq 4\,\sqrt{\ln (4n/\delta)/d}$ with probability $\geq
1-\delta/2$. For this we will use the concentration result,
\corref{freedman}. One can write:
$$\typL - \sM = \sum_{i\leq j}T_{\rm typ}X_{ij}T_{\rm typ}$$
where the $X_{ij}$ are the same matrices from the first part of the
proof (cf. \eqnref{adjacencyasmartingale}).  Again we have a sum of
mean-$0$ independent random matrices, in this case:
$$Y_{ij}\equiv T_{\rm typ}X_{ij}T_{\rm typ} = (I_{ij}-{\bf p}(i,j))\,\frac{A_{ij}}{\sqrt{\deg_{\typG}(i)\deg_{\typG}(j)}},\mbox{ with $A_{ij}$ as in \eqnref{defAe}.}$$
In all possible cases, the eigenvalues of $Y_{ij}$ are contained in
the set:
$$\left\{\frac{\pm(1-{\bf p}(i,j))}{\sqrt{\deg_{\typG}(i)\deg_{\typG}(j)}},\frac{\pm{\bf p}(i,j)}{\sqrt{\deg_{\typG}(i)\deg_{\typG}(j)}},0\right\}$$
and therefore $$\|Y_{ij}\|\leq
1/\sqrt{\deg_{\typG}(i)\deg_{\typG}(j)}\leq 1/d.$$ The sum of
variances is:
\begin{eqnarray*}\sum_{i\leq j}\Ex{Y_{ij}^2} &=& \sum_{i\leq j}\Ex{\left(I_{ij}-{\bf
p}(i,j)\right)^2}\left(\frac{A_{ij}}{\sqrt{\deg_{\typG}(i)\deg_{\typG}(j)}}\right)^2
\\ \mbox{(use \eqnref{Aijsquared})}&=& \sum_{i<j}{\bf p}(i,j)(1-{\bf p}(i,j))\,\frac{\cb_i\cb_i^* +
\cb_j\cb_j^*}{\deg_{\typG}(i)\deg_{\typG}(j)} \\ & & + \sum_{i}{\bf
p}(i,i)(1-{\bf p}(i,i))\,\frac{\cb_i\cb_i^*}{\deg_{\typG}(i)^2}\\
&=&
\sum_{i=1}^n\frac{1}{\deg_{\typG}(i)}\left(\sum_{j=1}^n\frac{{\bf
p}(i,j)(1-{\bf
p}(i,j))}{\deg_{\typG}(j)}\right)\,\cb_i\cb_i^*\end{eqnarray*} Again
we have a diagonal matrix. Its $(i,i)$-th entry is at most:
$$\frac{1}{\deg_{\typG}(i)}\left(\sum_{j=1}^n\frac{{\bf
p}(i,j)}{d}\right) = \frac{1}{d}.$$ We may thus apply
\corref{freedman} to $\sum_{ij}Y_{ij}$ with $M=\sigma^2=1/d$ to
obtain:
$$\Pr{\|\typL-\sM\|\geq t}\leq 2n\,e^{-\frac{t^2\,d}{8 +
4t}}.$$ To finish the proof, we take:
$$t = 4\sqrt{\frac{\ln(4n/\delta)}{d}}.$$
We have already ensured that $t\leq 3/4\leq 2$. This implies
$$\Pr{\|\typL-\sM\|\geq 4\sqrt{\frac{\ln(4n/\delta)}{d}}}\leq 2n\,e^{-\frac{16\ln(4n/\delta)}{16}} \leq \frac{\delta}{2}.$$This was precisely the required bound.\end{proof}

\begin{remark}[Comparing concentration bounds]\label{rem:comparisons}We now explain why the Hoeffding bound of Christofides and Markstr\"{o}m \cite{ChristofidesMarkstrom_HoeffdingForMatrices} is insufficient for our purposes. In the case of the adjacency matrix, the random sum we deal with is $\sum_{ij}(I_{ij}-{\bf p}(i,j))A_{ij}$. We observed above that $A_{ij}$ has eigenvalues $1$, $-1$ and $0$, hence we would have to take $r_i=1/2$ in order to apply \thmref{hoeffding} to $(\Ap-\typA)/2$. A simple calculation shows that the exponent in that bound would be of the order
$-t^2/\binom{n}{2}$ for small enough $t$, which is much worse than
the $-t^2/\Delta$ behavior we obtain. Our improvement comes from the
fact that our ``variance" term is the largest eigenvalue of a sum,
{\em not} the sum of largest eigenvalues. Similar comments apply to
the concentration of the Laplacian.\end{remark}

\section{The Erd\"{o}s-R\'{e}nyi graph and quasi-randomness}\label{sec:quasirandom}

As a first illustration of \thmref{typicalmatrices}, we apply our
results to the Erd\"{o}s-R\'{e}nyi graphs. Our bounds are suboptimal
in this very special case, but the stronger results in
\cite{FurediKomlos_Eigenvalues,FeigeOfek_Spectral} require more
difficult arguments that do not seem to generalize to other cases of
bond percolation (cf. \secref{bondpercolation}). Moreover, our
result correctly predicts the range of $p$ for which one can expect
concentration of the adjacency matrix.

We then connect concentration to the theory of {\em quasi-randomness
for dense graphs} \cite{ChungGrahamWilson_QuasirandomGraphs} showing
that, in a certain sense, quasi-randomness is equivalent to
concentration of the adjacency matrix.

While we will not dwell on this point, a similar connection could be
presented between {\em random graphs with given expected degrees}
\cite{ChungLuVu_SpectraGivenDegrees} and concentration of the
Laplacian. Our bounds are also suboptimal in this setting, as
attested by a recent preprint of Coja-Oghlan and Lanka
\cite{CojaLanka_SpectrumGivenDegrees}.

\subsection{Concentration for the Erd\"{o}s-R\'{e}nyi
graph}\label{sec:concentrationerdosrenyi}

For $0<p<1$, the Erd\"{o}s-R\'{e}nyi graph $G_{n,p}$
\cite{Bollobas_RandomGraphs,AlonSpencer_ProbMethod} is the special
case of the model $\Gp$ in \secref{typicalmatrices} where ${\bf
p}(i,j)=p$ for $i\neq j$ and ${\bf p}(i,i)=0$ for $i=j$. Notice that
in this case $\typA = p({\bf 1}_n{\bf 1}_n^* - I_n)$ where ${\bf
1}_n\in\R^n$ is the all-ones vector and $I_n$ is the $n\times n$
identity matrix. Moreover, $\typL = I_n - {\bf 1}_n{\bf 1}_n^*/n$

The following result is immediate from \thmref{typicalmatrices}

\begin{proposition}\label{prop:erdosrenyi}There exists $C>0$ such that for all $n\in\N$, $n^{-2}<\delta<1/2$ and $p\in(0,1)$ with $p(n-1)\geq C\ln n$, if
$A_{n,p}$ be the adjacency matrix and $\sL_{n,p}$ the Laplacian of
the Erd\"{o}s-R\'{e}nyi graph $G_{n,p}$, then
$$\Pr{\|A_{n,p} - p({\bf 1}_n{\bf 1}_n^* - I_n)\|\leq 4\, \sqrt{p(n-1)\ln
(n/\delta)}}\geq 1- \delta$$
$$\Pr{\|\sL_{n,p} - (I_n - {\bf 1}_n{\bf 1}_n^*/n)\|\leq 14\, \sqrt{\frac{\ln
(4n/\delta)}{p(n-1)}}}\geq 1- \delta$$

\end{proposition}

This result is qualitatively sharp in the sense that one cannot
expect that the Laplacian concentrates when $pn\ll \ln n$. To see
this, recall that the multiplicity of $0$ in the spectrum of
$\sL_{n,p}$ is the number of connected components of $G_{n,p}$ (this
is a deterministic statement; cf. \cite{Chung_SpectralGraphTheory}).
If $pn\leq \ln n$, the probability of there being $2$ or more
components is bounded away from $0$ \cite{Bollobas_RandomGraphs}.
But if $0$ has multiplicity $\geq 2$, \eqnref{eigenperturbation}
implies that $\|\sL_{n,p} - (I_n - {\bf 1}_n{\bf 1}_n^*/n)\|\geq 1$,
therefore $\sL_{n,p}$ is far from the ``typical Laplacian" with
positive probability.

Quantitatively, the bounds in \propref{erdosrenyi} can be improved.
We quickly sketch the argument for the adjacency matrix, which is
implicit in the work of Feige and Ofek \cite{FeigeOfek_Spectral}. A
key idea is that, since the typical adjacency matrix $p({\bf
1}_n{\bf 1}_n^* - I_n)$ has one very large eigenvalue and lots of
small ones, the same should hold for $A_{n,p}$.

One can use the reasoning in \cite[Lemma 2.1]{FeigeOfek_Spectral} to
show that, for $pn=\ohmega{\ln n}$ the dominant eigenvector of
$A_{n,p}$ is always close to ${\bf 1}_n/\sqrt{n}$. Moreover, the
largest eigenvalue is $pn + \bigoh{\sqrt{pn}}$ and all other are of
the order $\bigoh{\sqrt{pn}}$
\cite{FeigeOfek_Spectral,FurediKomlos_Eigenvalues}. This shows that,
with probability $\geq 1-\delta/2$
$$\|A_{n,p} - p{\bf 1}_n{\bf 1}_n^*\| = \bigoh{\sqrt{pn}},$$
and this results in
$$\|A_{n,p} - p({\bf 1}_n{\bf 1}_n^*-I_n)\| = \bigoh{\sqrt{pn}}\mbox{ with probability }\geq 1-\delta$$
because $\|pI_n\|=\bigoh{1}$.

\subsection{Quasi-randomness as concentration of the adjacency matrix}\label{sec:explanation}

We now point out that the idea of concentration of the adjacency
matrix is implicit in the theory of {\em dense quasi-random graphs}

This theory was initiated by Chung, Graham and Wilson
\cite{ChungGrahamWilson_QuasirandomGraphs}. Their surprising
discovery was that several properties that a Erd\"{o}s-R\'{e}nyi random graph is very likely to have are in fact equivalent.

More precisely, let $\{G_m\}_{m\in\N}$ be a sequence of graphs, each
$G_m$ having $n_m$ vertices and adjacency matrix $A_m$. Assume that
$n_m\to +\infty$ when $m\to+\infty$ and that $p>0$ is fixed. The
following statements (among others) are equivalent
\cite{ChungGrahamWilson_QuasirandomGraphs,KrivelevichSudakov_Pseudorandom}.
[The asymptotic notation refers to $m\to +\infty$.]

\begin{itemize}
\item {\bf [Q1]} There exists a $s\geq 4$ such that for all $0\leq k\leq \binom{s}{2}$, $G_m$ contains more than $p^{-k}(1-p)^{\binom{s}{2}-k}n_m^s$ induced labeled copies of each graph on $s$ vertices and $k$ edges.
\item {\bf [Q2]} $G_m$ has $\geq (1+\liloh{1})pn_m^2/2$ edges and $\leq (1+\liloh{1})(pn_m)^4$ labeled copies of the four-cycle $C_4$.
\item {\bf [Q3]} $G_m$ has $\geq (1+\liloh{1})pn_m^2/2$ edges, the largest eigenvalue of $A_m$ is $(1+\liloh{1})pn$ and all other eigenvalues of $A_m$ are $\liloh{n}$ in absolute value.
\item {\bf [Q4]} $\max_{S\subset V_m}|e(S)-p|S|^2/2|=\liloh{n_m^2}$ where $e(S)$ is the number of edges of $G_m$ inside $S$ and $V_m$ is the vertex set of $G_m$.\end{itemize}

We now provide a characterization of quasi-randomness in terms of
``concentration" of the adjacency matrix. Let
$$\typER\equiv p\,({\bf 1}_{n_m}{\bf 1}_{n_m}^* - I_{n_m}),$$
where ${\bf 1}_{n_m}\in\R^{n_m}$ is (again) the all-ones vector and
$I_{n_m}$ is the $n_m\times n_m$ adjacency matrix. This is the same
matrix that appears in \propref{erdosrenyi}.

The following result shows that a sequence of graphs is quasi-random
if and only if the adjacency matrices of the graphs are sufficiently
close to $\typER$.

\begin{proposition}\label{prop:quasirandom}A sequence $\{G_m\}_{m}$ of graphs as above satisfies properties {\bf [Q1]}-{\bf [Q4]} above if and only
if: $${\bf [P1]} \|A_m - \typER\|=\liloh{n}.$$\end{proposition}

\begin{proof}[of \propref{quasirandom}] We will show that {\bf [P1]} is equivalent to {\bf [Q3]} in the previous
list. \\

{\bf [P1]}$\Rightarrow${\bf [Q3]} : The eigenvalues of $\typER$ are
$p(n_m-1)$ (with multiplicity $1$) and $-p$ (with multiplicity
$n_m-1$).

We use inequality \eqnref{eigenperturbation} above to deduce that:
$$|\lambda_{n-1}(A_m)-pn_m| =  |\lambda_{n-1}(A_m)-\lambda_{n-1}(\typER)| +
\bigoh{1} = \liloh{n_m}$$ and for $0\leq i\leq n_m-2$:
$$|\lambda_{i}(A_m)| = |\lambda_{i}(A_m)+p| + \bigoh{1} =
|\lambda_{i}(A_m)-\lambda_{i}(\typER)| + \bigoh{1} = \liloh{n_m}.$$
Moreover, the number of edges in $G_m$ is:
\begin{eqnarray*}\frac{1}{2}\,{\bf 1}^*_{n_m}A_{m}{\bf 1}_{n_m}&\geq& {\bf
1}^*_{n_m}\typER {\bf 1}_{n_m} - \frac{1}{2}\,{\bf
1}^*(A_m-\typER)\,{\bf 1} \\ &=& \frac{pn_m(n_m-1)}{2} - \|{\bf
1}_{n_m}\|^2 \|A_m-A_{K_{n_m}}\| \\ &=& \frac{pn_m^2}{2} -
\liloh{n_m^2}.\end{eqnarray*}

{\bf [Q3]}$\Rightarrow ${\bf [P1]}: It is immediate from {\bf [Q3]}
that $A_m$ is $\liloh{n}$-close to a rank-one operator: if
$\psi_{\max}$ is the (normalized) eigenvector corresponding to the
largest eigenvalue $\lambda_{\max}(A_m)$, then:
$$\|A_m - \lambda_{\max}(A_m)\,\psi_{\max}\psi_{\max}^*\| = \max_{0\leq i\leq
n_m-1}|\lambda_{i}(A_m)| = \liloh{n_m}.$$ By {\bf [Q3]} we also know
that:
$$\|\lambda_{\max}(A_m)\,\psi_{\max}\psi_{\max}^* -
pn_m\,\psi_{\max}\psi_{\max}^*\| =
|\lambda_{\max}(A_m)-pn_m|=\liloh{n}.$$ It is shown in the proof of
Fact 7 in \cite{ChungGrahamWilson_QuasirandomGraphs} that, under
{\bf [Q3]}, $\psi_{\max}$ is $\liloh{1}$-close to ${\bf
1}_{n_m}/\sqrt{n_m}$. Thus we see that:
$$\|pn_m\,\psi_{\max}\psi_{\max}^* - p{\bf 1}_{n_m}\,{\bf
1}^*_{n_m}\| = \liloh{n_m}.$$ Finally, we notice that

$$\typER = p{\bf 1}_{n_m}\,{\bf
1}^*_{n_m} - pI,$$ hence
$$\|\typER -p{\bf 1}_{n_m}\,{\bf
1}^*_{n_m}\|=\bigoh{1}.$$ Putting all the inequalities together
implies the desired result.\end{proof}

\section{Application to bond percolation}\label{sec:bondpercolation}

In the previous section we discussed a random graph model where the
typical Laplacian and adjacency matrices had one ``special"
eigenvalue with multiplicity $1$ and $n-1$ ``trivial" eigenvalues.
In this setting, proving concentration of the adjacency matrix (say)
essentially amounted to showing that one eigenvector was close to
what it should be while the other eigenvalues clustered around the
degenerate eigenvalue of the typical case.

We now consider a class of models for which one cannot expect this
strategy to work. Let $p\in(0,1)$ and $G=(V,E)$ be an arbitrary
unweighted graph on vertex set $V=[n]$. Consider the random subgraph
$G_p$ of $G$ that is obtained via by deleting each edge of $G$
independently with probability $1-p$. This model of {\em bond
percolation} has received much attention in recent years, with a
special focus the emergence of a giant component
\cite{BollobasEtAl_PercolationOnDense,FriezeKrivelevichMartin_GiantPseudorandom,AlonBenjaminiStacey_PercolationOnFinite,ChungLiuHorn_Percolation,Nachmias_MeanFieldPercolation,BorgsEtAl_PercolationTriangle}.
Much less seems to be known about the spectrum of $G_p$
\cite{ChungHorn_PercolationSpectrum}.

In this section we apply our general Theorem,
\thmref{typicalmatrices}, in order to answer the following question:
how large does $p$ need to be in order for the graph matrices to
concentrate? Clearly, this must occur way after the percolation
threshold.

Bond percolation is a special case of the random model $\Gp$ in
\secref{typicalmatrices}. To see this, one only needs to define:
$${\bf p}(i,j) = \left\{\begin{array}{ll} p & \mbox{if $ij\in E$,}\\ 0 & \mbox{ if not.}\end{array}\right.$$

A computation shows that the ``typical matrices" for this choice of
${\bf p}$ are:

$$\typA = pA_G, \mbox{ where $A_G$ is the adjacency matrix of
$G$;}$$
$$\typL = \sL_G, \mbox{ where $\sL_G$ is the Laplacian of
$G$.}$$

Moreover, the parameters $d$, $\Delta$ appearing in
\thmref{typicalmatrices} are $pd_G$ and $p\Delta_G$, where $d_G$
(resp. $\Delta_G$) is the minimum (resp. maximal) degree in $G$.

The following result is a direct corollary of
\thmref{typicalmatrices}.

\begin{theorem}\label{thm:bondpercolation}For each $c>0$ there exists a $C>0$ such that the following holds. Suppose that
$G$, $p$ and $G_p$ are as above and $pd_G\geq C\,\ln n$. Then:
$$\Pr{\|A_{G_p} - pA_G\|\leq 4\,\sqrt{p\Delta_G\,\ln (n/\delta)}}\geq 1-\delta$$
and
$$\Pr{\|\sL_{G_p} - \sL_G\|\leq 14\,\sqrt{\frac{\ln(4n/\delta)}{pd_G}}}\geq 1-\delta,$$
where $A_{G_p}$ and $\sL_{G_p}$ are the adjacency matrix and
Laplacian of $G_p$ (resp.)\end{theorem}

One can of course derive corollaries about eigenvectors and eigenvectors following Corollaries \ref{cor:typicaleigenvalues} and \ref{cor:typicaleigenvectors}. For
instance, suppose that:
$$\gamma> 14\sqrt{\frac{\ln(4n/\delta)}{pd_G}}$$
Then the following holds with probability $1-\delta$: for each $0\leq i\leq n-1$ such that the interval
$(\lambda_i(\sL_G)-2\gamma,\lambda_i(\sL_G)+2\gamma)$ contains no
eigenvalues of $\sL_G$ other than $\lambda_i(\sL)$,
$\lambda_i(\sL_{G_p})$ has multiplicity $1$ in the spectrum of
$\sL_{G_p}$ and moreover, the corresponding normalized eigenvectors $\psi$,
$\psi_p$ of $\sL_{G}$ and $\sL_{G_p}$ (resp.) satisfy:
$$\|\psi_p\psi_p^* - \psi\psi^*\|\leq
\frac{4}{\pi}\,\frac{\sqrt{\frac{\ln(4n/\delta)}{pd_G}}}{\gamma -
\sqrt{\frac{\ln(4n/\delta)}{pd_G}}}$$ with probability $\geq
1-\delta$. This implies:
$$1-(\psi^*\psi_p)^2\leq \frac{4}{\pi}\,\frac{\sqrt{\frac{\ln(4n/\delta)}{pd_G}}}{\gamma -
\sqrt{\frac{\ln(4n/\delta)}{pd_G}}}$$ for the same eigenvectors,
which implies that $\psi_p$ is close to $\psi$ or $-\psi$. A similar
result for the eigenspace projectors could be derived even if
$\lambda_i(G)$ had higher multiplicity. It seems quite remarkable
that one can approximately obtain the eigenvectors or eigenspaces of
$G$ from a (potentially very sparse) subgraph $G_p$.

We also note that the threshold for Laplacian concentration is
indeed $pd_G=\theita{\ln n}$, as shown in
\secref{concentrationerdosrenyi} in the special case of the
Erd\"{o}s-R\'{e}nyi random graph $G_{n,p}$.

The following simple corollary is also of interest.

\begin{corollary}There exist $C,C'>0$ such that, if $pd_G\geq C\ln n$, then with probability $1-1/n^2$,
$$|\lambda(G)-\lambda(G_p)|\leq C'\,\sqrt{\frac{\ln n}{pd_G}}.$$\end{corollary}

We have singled out this bound in order to compare it with a recent
bound of Chung and Horn \cite{ChungHorn_PercolationSpectrum}. These
authors proved that, with high probability,
$$\lambda(G_p)\geq \lambda(G) - \bigoh{\sqrt{\frac{\ln n}{pd_G}} + \frac{(\ln n)^{3/2}}{pd_G (\ln\ln
n)^{3/2}}}.$$ Our bound is better for all values of $n$ and $pd_G$,
most dramatically for $\ln n\ll pd_G\ll \ln^{3/2-\eps}n$, in which
case their bound is vacuous while ours is non-trivial.

\section{Application to inhomogeneous random graphs}\label{sec:inhomogeneous}

In this section we consider a more complex random graph model that
is defined in terms of an {\em attachment kernel} $\kappa$, a {\em
density parameter} $0<p<1$ and a set of points
$X_{1},X_2,\dots,X_n$.

More precisely, let $\kappa:[0,1]^2\to \R_+\cup\{0\}$ be a measurable
function that is symmetric in the sense that for all $x,y\in[0,1]$,
$\kappa(x,y)=\kappa(y,x)$. Pick some vector $X_{1:n}\equiv
(X_1,\dots,X_n)$ of points in $[0,1]$. Now consider the following
weight function ${\bf p}:[n]^2\to[0,1]$:

\begin{equation}\label{eq:defpinhomogeneous}{\bf p}(i,j)\equiv \max\{p\kappa(X_i,X_j),1\},
\,(i,j)\in[n]^2.\end{equation}

One can define a random graph $\Gp$ as in \secref{typicalmatrices}
with the above weight function; we call this graph $G_{n,p,\kappa}$,
the {\em inhomogeneous random graph} on $n$ vertices, density
parameter $p$ and attachment kernel $\kappa$ (the dependency on
$X_{1:n}$ is implicit in this nomenclature). The adjacency matrix of this random graph will be denoted by $A_{n,p,\kappa}$

Our goal in this section will be to prove that, up to some error terms that are small with high probability, the
adjacency matrix $A_{n,\kappa,p}/pn$ of $G_{n,p,\kappa}$ will be related to the integral operator on $L^2([0,1])$ that is defined by $\kappa$.
\begin{equation}\label{eq:integralop}\begin{array}{lllll}T_\kappa
&:& L^2([0,1]) &\to& L^2([0,1])\\ & & f(\cdot) &\mapsto & \int_{0}^1
\kappa(\cdot,y)\,f(y)\,dy\end{array}\end{equation}

Similar results for the Laplacian of $G_{n,p,\kappa}$ are discussed
in \secref{final}.

\subsection{Some history of the model}

The phrase ``inhomogeneous random graph" comes from a paper by
Bollobas, Janson and Riordan
\cite{BollobasEtAl_InhomogeneousRandomGraphs} where the above model
was studied in the range $p=\theita{1/n}$ with background spaces
more general than $[0,1]$. Their goal was to study the structure of
connected components in the general model, in analogy with the
well-known Erd\"{o}s-R\'{e}nyi phase transition at $p=1/n$
\cite{AlonSpencer_ProbMethod}.

A related random graph model generating dense graphs ($p=1$) was
introduced in \cite{LovaszSzegedy_GraphLimits} and studied in
\cite{BorgsEtAl_ConvergentSequencesDenseGraphs}. This model is
related to the beautiful theory of {\em graph limits} where the
space of graphs is ``completed" into the space of {\em graphons},
which are non-negative, symmetric functions like $\kappa$ above,
with the further restriction that $\kappa\leq 1$. There is a fairly
complete correspondence between the properties of sequences of
graphs that are {\em convergent} in terms of normalized subgraph
counts and the corresponding limiting graphon. Conversely, the
sequence of random graphs correponding to a given graphon $\kappa$
converges to that same graphon. The {\em cut metric} that defines
graph convergence will be further discussed in \secref{cutmetric}
below.

The connection between the convergent graph sequences and
inhomogeneous random graphs was noted in
\cite{BollobasEtAl_PercolationOnDense,BollobasEtAl_CutMetric}, where
the authors studied bond percolation over a convergent sequence of
graphs and found the critical probability for existence of a giant
component. Other papers
\cite{BollobasRiordan_MetricsSparseGraphs,BollobasRiordan_VerySparseGraphs}
have focused on the relationship between convergence of subgraph
counts vs. convergence in the {\em cut metric} (see below) for
sparse graphs, a topic that is far from completely elucidated. In
what follows we will show that our random graphs converge to the
corresponding kernel in a stronger metric.

\subsection{The precise result}\label{sec:preciseresult}
We will use the following technical assumption.

\begin{assumption}\label{assum:kappa}$\kappa:[0,1]^2\to \R_+\cup\{0\}$ is a symmetric measurable function with
$$K\equiv\sup_{(x,y)\in[0,1]^2}\kappa(x,y)<+\infty.$$ Moreover, the points $X_1,X_2,\dots,X_n$ are random i.i.d. uniform over $[0,1]$.\end{assumption}

Let $\overline{X}_1\leq \overline{X}_2\leq \dots\leq \overline{X}_n$ be the ordered sequence of the $X_i$; ie. $\overline{X}_1$ is the minimum of the $X_i$, $\overline{X}_2$ is the second smallest element and so on (ties are broken arbitrarily). Let $\sigma_n$ be a permutation such that $X_{i}=\overline{X}_{\sigma_n(i)}$ for each $1\leq i\leq n$, chosen in a measurable manner. Associate with the graph $G_{n,p,\kappa}$ a symmetric, non-negative function from $[0,1]^2$ to $\R_+\cup\{0\}$:
$$\sG_{n,p,\kappa}\equiv \frac{1}{p}\,\sum_{ij\in E(G_{n,p,\kappa})} \Ind{\left(\frac{\sigma_n(i)-1}{n},\frac{\sigma_n(i)}{n}\right]\times \left(\frac{\sigma_n(j)-1}{n},\frac{\sigma_n(j)}{n}\right]},$$
where $\Ind{S}$  is the indicator function of the set $S$ and $E(G_{n,p,\kappa})$ is the edge set of $G_{n,p,\kappa}$. Notice that $\sG_{n,p,\kappa}$ defines a bounded linear operator on $L^2([0,1])$ via a formula similar to \eqnref{integralop}:
$$(T_{\sG_{p,n,\kappa}}f)(\cdot)\equiv \int_0^1 \sG_{p,n,\kappa}(\cdot,y)\,f(y)\,dy\;\;(f\in L^2([0,1])).$$

Let $\{\cb_i\}_{i=1}^n$ be the canonical basis of $\R^n$. Let us
consider two linear operators (both of which depend on $\sigma_n$
defined previously:
$$\begin{array}{lllll}H_n & : & \R^n &\to &L^2([0,1])\\
& & \psi=\sum_{i=1}^n\psi(i)\,\cb_i &\mapsto & \sum_{i=1}^n \sqrt{n}\,\psi(i)\,\Ind{\left(\frac{\sigma_n(i)-1}{n},\frac{\sigma_n(i)}{n}\right]}\\ & & & & \\ E_n & : & L^2([0,1]) &\to &\R^n\\
 & & f &\mapsto & \sum_{i=1}^n\left(\sqrt{n}\int_{\frac{\sigma_n(i)-1}{n}}^{\frac{\sigma_n(i)}{n}}\,f\,dx\right)\,\cb_i\end{array}$$
 and note that $T_{\sG_{n,p,\kappa}} = H_n A_{n,p,\kappa} E_n/pn$.

Finally, let ${\rm spec}(T_\kappa)$ be the {\em spectrum} of the operator $T_\kappa$ in \eqnref{integralop} (see \secref{integralop} to recall what the spectrum is).

\begin{theorem}[proven in \secref{proofinhomogeneous}]\label{thm:inhomogeneous} There exist universal constants $c,C>0$ such that the following holds under \assumref{kappa}. Given $\eps>0$, suppose there exists a $L$-Lipschitz function $\kappa_{\eps}$ that also takes values in $[0,K]$ and which is $\eps$-close to $\kappa$ in the $L^2([0,1]^2)$ norm. Define:
$$\theta=\theta(\kappa,\eps,L,K,n,p)\equiv 2\eps + c(L+K)\left(\frac{\ln n}{n}\right)^{1/4} + \sqrt{\frac{K\ln n}{pn}},$$
and assume $pn\geq C\ln n$ and $p\leq 1/K$. Then there exists an
event $\sE$ with probability $\Pr{\sE}\geq 1-n^{-2}$ such that,
inside $\sE$, the following properties hold:
\begin{enumerate}
\item The $n\times n$ matrices $A_{n,p,\kappa}$ and $E_n T_\kappa H_n$ satisfy:
$$\left\|\frac{A_{n,p,\kappa}}{pn} - E_n T_\kappa H_n \right\|\leq \theta;$$
\item The integral operators $T_{\sG_{n,p,\kappa}}$ and $T_\kappa$ satisfy:
$$\left\|T_{\sG_{n,p,\kappa}} - T_\kappa \right\|_{2\to 2}\leq \theta;$$
\item Given $S\subset\R$, let $m_{A_{n,\kappa,p}/pn}(S)$ be the sum of the multiplicities of all eigenvalues of $A_{n,\kappa,p}$ that lie in $S$ and define $m_{T_\kappa}(S)$ similarly. Then if $\inf_{s\in S}|s|>\theta$,
$$m_{A_{n,\kappa,p}/pn}(S)\leq m_{T_\kappa}(S^\theta)\mbox{ and }m_{T_\kappa}(S)\leq m_{A_{n,\kappa,p}/pn}(S^\theta)
$$ where $S^\theta\equiv \{x\in \R\,:\, \exists s\in S,\,|x-s|\leq \theta\}$.
\item Consider each pair $(\alpha,\gamma)$ where $\alpha\in {\rm spec}(T_\kappa)$ and $\gamma>\theta$ is such that $(\alpha-2\gamma,\alpha+2\gamma)$ contains no eigenvalue of $T_\kappa$ other than $\alpha$ itself. Let $P_\alpha$ be the orthogonal projection in $L^2([0,1])$ onto the eigenspace of $\alpha$ in $L^2([0,1])$ and consider the orthogonal projection $\Pi_{(\alpha-\gamma)pn,(\alpha+\gamma)pn}(A_{n,p,\kappa})$ in $\C^n$ over the span of the eigenvectors of $A_{n,p,\kappa}$ corresponding to eigenvalues in $[(\alpha-\gamma)pn,(\alpha+\gamma)pn]$. Then:
    $$\|\Pi_{(\alpha-\gamma)pn,(\alpha+\gamma)pn}(A_{n,p,\kappa}) - E_n P_\alpha H_n\|\leq \frac{4\theta}{\pi(\gamma-\theta)}.$$
\end{enumerate}
\end{theorem}

This Theorem implies that, up to error terms that are small with
high probability, $A_{n,p,\kappa}/pn$ is defined
solely in terms of the kernel function $\kappa$, up to a permutation of coordinates. It also implies that, statistical
parlance, it implies that the non-zero eigenvalues of $A_{n,p,\kappa}$ are
strongly consistent estimators of the non-zero eigenvalues of $T_\kappa$ when
$n\to +\infty$ and $p=p(n)$ $pn/\ln n\to +\infty$.

Both of these assertions hinge on the fact that Lipschitz functions
are dense in $L^2([0,1]^2)$. Unfortunately, our error bounds are not
independent of $\kappa$, as quality of the approximation by
Lipschitz functions, measured by the size of the Lipschitz constant for a given approximation error $\eps$, may vary with $\kappa$. This is in contrast with
approximation in the {\em cut norm}, which we now discuss.

\subsection{Convergence in the operator and cut metrics}\label{sec:cutmetric}

\subsubsection{The cut norm and the cut metric}

Any function $\eta\in L^1([0,1]^2)$ determines a bounded linear
operator $\tilde{T}_\eta:L^{\infty}([0,1])\to L^{1}([0,1])$ via the formula that we already used to define $T_\kappa$ and $T_{\sG_{p,n,\kappa}}$:
\begin{equation}\label{eq:Teta}\tilde{T}_\eta f(\cdot) \equiv \int_{0}^1
\eta(\cdot,y)\,f(y)\,dy.\end{equation} The cut norm of $\eta$ is the
$L^\infty\to L^1$ norm of $\tilde T_\eta$:
\begin{equation}\label{eq:cutnorm}\|\eta\|_{\rm cut}\equiv
\sup\left\{\left|\int_0^1 (T_\eta f)(x)g(y)\,dx\right|\,:\, f,g\in
L^{\infty}([0,1]), \|f\|_{L^\infty}\leq 1,\|g\|_{L^\infty}\leq
1\right\}.\end{equation} One can check that $\|\eta\|_{\rm cut}\leq
\|\eta\|_{L^1}$ always. This definition of $\|\eta\|_{\rm cut}$ is natural from
the point of view of Functional Analysis; a more ``combinatorial"
definition,
$$\|\eta\|_{\rm cut,2}\equiv \sup\left\{\left|\int_{A\times B}
\eta(x,y)\,dx\,dy\right|\,:\, A,B\subset [0,1]\mbox{
measurable}\right\}$$ is equivalent to the previous one in the sense that:
$$\frac{1}{4}\|\eta\|_{\rm cut}\leq \|\eta\|_{\rm cut,2}\leq \|\eta\|_{\rm cut}.$$

Now assume that $G_1$ and $G_2$ are graphs with common vertex set $[n]$ and adjacency matrices $A_{G_1},A_{G_2}$. Define:
$$\kappa_{G_i,p}\equiv \sum_{1\leq i,j\leq n\,:\, ih\in E(G_i)} \Ind{\left(\frac{i-1}{n},\frac{i}{n}\right]\times \left(\frac{j-1}{n},\frac{j}{n}\right]}.$$
Then one sees that: $$\|\kappa_{G_1,p}-\kappa_{G_2,p}\|_{\rm cut,2} =
\frac{\max_{S,V\subset[n]}\left|\sum_{(i,j)\in S\times
W}(A_{G_1}(i,j)-A_{G_2}(i,j))\right|}{pn^2}$$ is the normalized {\em cut
norm} of $A_G-A_H$ \cite{LovaszSzegedy_GraphLimits}.

Thus the cut norm on $L^1([0,1]^2)$ induces a distance on graphs.
Notice, however, that this distance might be positive even though
$G$ and $H$ are isomorphic. This motivates the following definition:
given two kernels $\kappa,\kappa''\in L^1([0,1]^2)$ , say that
$\kappa''$ is a rearrangement of $\kappa$ ($\kappa''\approx \kappa$)
is there exists a measure-preserving bijection $\tau:[0,1]\to [0,1]$
such that $\kappa(x,y)=\kappa(\tau(x),\tau(y))$ for almost every
$(x,y)\in [0,1]^2$. The {\em cut metric} assigng to each pair
$\kappa,\kappa'$ of kernels a distance:
$$d_{\rm cut}(\kappa,\kappa')\equiv \inf\{\|\kappa''-\kappa'\|_{\rm
cut}\,:\, \kappa''\approx \kappa\}.$$ Notice that the cut metric
does not distinguish between (the kernels of) isomorphic graphs.

\subsubsection{The operator norm and the operator metric}

The metric $d_{\rm cut}$ yields a criterion for convergence of graph
sequences. In the dense case $p=\theita{1}$, this implies the
convergence of normalized subgraph counts and also gives a criterion
for testable graph properties
\cite{LovaszSzegedy_GraphLimits,BorgsEtAl_ConvergentSequencesDenseGraphs}.
As mentioned above, much less is understood about the case
$p=\liloh{1}$ (see however the conjectures of Bollob\'{a}s and
Riordan \cite[Section 5.2]{BollobasRiordan_MetricsSparseGraphs}).

\thmref{inhomogeneous} is mostly concerned with the
eigenvalues and the eigenvectors of the adjacency matrix $A_{n,p,\kappa}$. Unfortunately, in general we do not even know how to control the eigenvalues of $A_{n,p,\kappa}$ in terms of the cut norm alone. For {\em bounded} kernels ($p=\theita{1}$), this is easy enough (see \cite[Theorem
6.6]{BorgsEtAl_ConvergentSequencesDenseGraphsII}), but there are difficulties in extending this to the
sparse case. This does seem to be a serious problem, as related difficulties appear in
\cite{BollobasRiordan_MetricsSparseGraphs} when the authors attempt
to relate the convergence of subgraph counts to cut metric
convergence. [Estimating the eigenvalues is related to
counting cycles in the corresponding graph or graphon.]

Luckily, a stronger notion of convergence implied by the $L^2\to L^2$ norm suffices for our purpose, and it is precisely this notion that we achieve via our methods.

We need some definitions in order to properly state this. Given
$\eta\in L^2([0,1]^2)$, define a bounded linear operator $T_\eta$
from $L^2([0,1])$ to itself via the formula in \secref{integralop}; this is the same as \eqnref{Teta}, except that the domain and range of $\tilde{T}_\eta$ are different. The
operator or ``$L^2\to L^2$" norm of $\eta$ is the $L^2\to L^2$ norm of
$T_\eta$, also defined in \secref{integralop}:
$$\|\eta\|_{\rm op} \equiv\|T_\eta\|_{L^2\to L^2}.$$ From \eqnref{cutnorm} we see that
that $\|\eta\|_{\rm op}\geq \|\eta\|_{\rm cut}$ whenever $\eta$ is square-integrable.

In analogy with the cut metric, one can also define an {\em operator
(pseudo-)metric} on square-integrable kernels via the formula:
$$d_{\rm op}(\kappa,\kappa')\equiv \inf\{\|\kappa''-\kappa'\|_{\rm op}\,:\,\kappa''\approx
\kappa\}.$$

One can show via our results that when \assumref{kappa} holds, $n\gg 1$ and $p\gg \ln n/n$, then the kernel determined by $G_{n,p,\kappa}$ -- which is equivalent to $\sG_{n,p,\kappa}$ in \thmref{inhomogeneous} -- converges in the $d_{\rm op}$ metric to $\kappa$. We omit the details.

A drawback of $d_{\rm op}$ is that it lacks a corresponding (weak or strong) regularity lemma, which would allow one to approximate up to error $\eps$ any (say bounded) kernel $\kappa$ by simple functions taking at most $m=m(\eps,\|\kappa\|_{L^\infty})$ values. Indeed, this is precisely why the bound in \thmref{inhomogeneous} depends on $\kappa$.

\subsection{Proof of \thmref{inhomogeneous}}\label{sec:proofinhomogeneous}

The proof will consist of several steps.

\subsubsection{The relationship between $T_{\sG_{n,p,\kappa}}$ and $A_{\sG_{n,p,\kappa}}$}

For $f,g\in L^2([0,1])$, define $(f,g)_{L^2}\equiv \int_0^1 f(x)g(x)\,dx$ and $\|f\|_{L^2}^2\equiv (f,f)_{L^2}$.

The following facts can be easily checked (proof omitted).
\begin{equation}\label{eq:isadjoint}\forall f\in L^2([0,1]),\, \forall \psi\in \R^n,\, (f,H_n\psi)_{L^2} = (E_n f)^*\psi \mbox{ (i.e. $E_n$ is the adjoint of $H_n$);}\end{equation}
\begin{equation}\label{eq:Hnhasnorm1}\forall \psi,\phi\in \R^n,\, (H_n\psi,H_n\phi)_{L^2} = \psi^*\phi \mbox{ (i.e. $H_n$ is an isometry)'}\end{equation}
\begin{equation}\label{eq:Pnhasnorm1}\forall f\in L^2,\, \|E_n f\| \leq \|f\|_{L^2} \mbox{ (i.e. $E_n$ has operator norm at most $1$);}\end{equation}
\begin{equation}\label{eq:identity}E_n H_n = I_n,\mbox{ the identity operator on $\R^n$;}\end{equation}
\begin{equation}\label{eq:projector}H_n E_n = \Pi_n,\mbox{ the projection onto the span of }\left\{\left(\frac{i-1}{n},\frac{i}{n}\right]\right\}_{i=1}^n; \mbox{ and}\end{equation}
\begin{equation}\label{eq:GandA}T_{\sG_{n,p,\kappa}} = \frac{1}{pn} H_n\,A_{n,p,\kappa}\,E_n, \mbox{ as seen above.}\end{equation}

Let us now relate the non-zero eigenvalues and eigenvectors of $A_{n,p,\kappa}$ with those of $T_{\sG_{n,p,\kappa}}$. \ignore{The first thing to notice is that:
$$T_{\sG_{n,p,\kappa}}= \Pi_n T_{\sG_{n,p,\kappa}}.$$ This implies that if $0\neq f\in L^2([0,1])$ is an eigenvector of $T_{\sG_{n,p,\kappa}}$ with eigenvalue $\alpha\neq 0$,
$$ \Pi_n f = \frac{\Pi_n T_{\sG_{n,p,\kappa}} f}{\alpha} = \frac{T_{\sG_{n,p,\kappa}}\,f}{\alpha} = f.$$
In other words, $f$ belongs to the range of $\Pi_n$ and is therefore constant on each interval $((i-1)/n,i/n]$. One can check that this implies that $\exists \psi\in\C^n$, $H_n\psi = f$. Using
\eqnref{GandA} and \eqnref{identity} we deduce that:
$$A_{n,p,\kappa} \psi = pn\, E_n T_{\sG_{n,p,\kappa}} H_n\psi = (\alpha pn)\, E_n H_n\psi = (\alpha pn)\psi.$$
Therefore, all eigenvectors of $T_{\sG_{n,p,\kappa}}$ with non-zero eigenvalue $\alpha$ are of the form $H_n\psi$ for some eigenvector $\psi$ of $A_{n,p,\kappa}$ with eigenvalue $\alpha pn$. It is easy to see that the converse holds.} Write:
$$A_{n,p,\kappa} = \sum_{\alpha:\alpha pn \in {\rm spec}(A_{n,p,\kappa})} (\alpha pn)\,\Pi_{\alpha}$$
where each $\Pi_{\alpha}$ the projection onto the eigenspace corresponding to $\alpha pn$. By \eqnref{GandA},
$$T_{\sG_{n,p,\kappa}}= \sum_{\alpha:\alpha pn \in {\rm spec}(A_{n,p,\kappa})} \alpha \,H_n\,\Pi_{\alpha}E_n.$$
\begin{claim}\label{claim:puf}The operators $H_n\,\Pi_{\alpha}E_n$ are orthogonal projections with orthogonal ranges. Therefore, the non-zero eigenvalues of $T_{\sG_{n,p,\kappa}}$ are the numbers $\alpha\neq 0$ with $\alpha pn \in{\rm spec}(A_{n,p,\kappa})$. Moreover, for each such $\alpha$, $H_n\,\Pi_{\alpha}E_n$ is the projection onto the corresponding eigenspace of $T_{\sG_{n,p,\kappa}}$.\end{claim}
\begin{proof}[of the Claim] First notice that for each $\alpha$:
$$(H_n\,\Pi_{\alpha}E_n)^2 = H_n\,\Pi_{\alpha}E_n H_n \Pi_{\alpha} E_n = H_n\,\Pi_{\alpha}E_n.$$
because $E_nH_n = I_n$ (eqn. \eqnref{identity}) and
$\Pi_{\alpha}^2=\Pi_{\alpha}$. One can also check that for all
$f,g\in L^2([0,1])$,
$$(f,H_n\,\Pi_{\alpha}E_n g)_{L^2} = (H_n f)^* (\Pi_{\alpha} E_n g) = (\Pi_{\alpha} H_n f)^* (E_n g) = (H_n\,\Pi_{\alpha}E_nf, g)_{L^2},$$
where we used \eqnref{isadjoint} for the first and third equalities
and the fact that $\Pi_{\alpha}=\Pi_\alpha^*$ for the second one. It
follows that $H_n\,\Pi_{\alpha}E_n$ is a self-adjoint operator on
$L^2$ that equals its square; this means that it is an orthogonal
projection onto its range.

To see that these ranges are orthogonal for distinct $\alpha$,
notice that the range of $H_n\,\Pi_{\alpha}E_n$ is the set of all
vectors of the form $H_n\psi$ where $\psi$ belongs to the range of
$\Pi_{\alpha}$ and is therefore an eigenvector of $A_{n,p,\kappa}$
with eigenvalue $\alpha pn$. But eigenvectors of $A_{n,p,\kappa}$
with distinct eigenvalues are orthogonal, hence their images under
$H_n$ are orthogonal in $L^2$ (by \eqnref{Hnhasnorm1}).

The other assertions follow directly.\end{proof}

\subsubsection{The concentration argument}

Let us introduce a matrix  $\overline{A}_{n,p,\kappa}$ whose
$(i,j)$-th entry is $p\kappa(X_i,X_j)$, $1\leq i\leq j\leq n$.
Conditioning on the realization of the $X_1,\dots,X_j$, our random
graph model has independent edges with respective probabilities
${\bf p}(i,j)=p\kappa(X_i,X_j)$ and $\overline{A}_{n,p,\kappa}$ is
precisely the typical adjacency matrix $\typA$ in this setting. We
deduce from \thmref{typicalmatrices} that there exists a constant
$C>0$ independent of $n,\kappa$ and $X_1,\dots,X_n$, such that if
$\Delta=\Delta(X_1,\dots,X_n)$ is as in that Theorem and $\Delta
\geq C\ln n$,
$$\Pr{\|A_{n,p,\kappa}-\overline{A}_{n,p,\kappa}\|\geq 4\,\sqrt{\Delta \ln(2n^2)}\mid X_1,\dots,X_n}\leq \frac{1}{2n^{2}},$$
In our setting we always have
$$\Delta= \max_{1\leq i\leq n}\sum_{j=1}^np\kappa(X_i,X_j)\leq Kpn$$
where $K$ is the quantity in \assumref{kappa}. Therefore,
$$\Pr{\|A_{n,p,\kappa}-\overline{A}_{n,p,\kappa}\|\geq 4\,\sqrt{Kpn\ln(2n^2)}}\leq \frac{1}{2n^{2}}.$$
Let
\begin{equation}\label{eq:Tbar}\overline{T}\equiv H_n\overline{A}_{n,p,\kappa}E_n/pn = \sum_{1\leq i,j\leq n}\kappa(X_i,X_j)\Ind{\left(\frac{\sigma_n(i)-1}{n},\frac{\sigma_n(i)}{n}\right]\times \left(\frac{\sigma_n(j)-1}{n},\frac{\sigma_n(j)}{n}\right]}\end{equation} Since $H_n$ is an isometry (by \eqnref{Hnhasnorm1})
and $E_n$ has norm at most $1$ (by \eqnref{Pnhasnorm1}),
$$\|\overline{T} - T_{\sG_{n,p,\kappa}}\|_{L^2\to L^2} = \frac{1}{pn}\|H_n(\overline{A}_{n,p,\kappa}-A_{n,p,\kappa})E_n\|\leq 4\,\sqrt{\frac{K\ln(2n^2)}{pn}}$$
with probability $\geq 1-1/2n^2$.

\subsubsection{Nearing the end of the argument}

We will show in \lemref{approx} below that there exists a universal $c>0$ such that for
any $\eps>0$
\begin{equation}\label{eq:goooooo}\Pr{\|\overline{T} - T_{\kappa}\|\leq 2\eps + c(L+K)(\ln
n/n)^{1/4}}\geq 1- \frac{1}{2n^2}.\end{equation} Increasing $c$ if
necessary, this implies that, with probability $\geq 1-n^{-2}$
$$\|T_{\sG_{n,p,\kappa}} - T_{\kappa}\|_{L^2\to L^2}\leq \|T_{\sG_{n,p,\kappa}} - \overline{T}\|_{L^2\to L^2}+ \|\overline{T}-T_\kappa\|_{L^2\to L^2}\leq \theta$$
for $\theta$ as in the Theorem. This proves the second assertion in
the Theorem. To prove the first one, first notice that, since
$E_nH_n=I_n$ (cf. \eqnref{identity}),
$$E_nT_{\sG_{n,p,\kappa}}H_n=\frac{1}{pn}\,(E_nH_n)A_{n,p,\kappa}(E_nH_n) = \frac{A_{n,p,\kappa}}{pn}.$$
Now use again the fact that $E_n$ and $H_n$ have norm $1$ to deduce:
$$\|\frac{A_{n,p,\kappa}}{pn} - E_nT_{\kappa}H_n\|\leq \|T_{\sG_{n,p,\kappa}} - T_{\kappa}\|_{L^2\to L^2}\leq \theta.$$
The other two assertions follow from the perturbation lemmas
provided in the Appendix. More precisely, recall from Claim
\ref{claim:puf} that the eigenvalues of $T_{\sG_{n,p,\kappa}}$ are
either $0$ or equal to some $\alpha\neq 0$ with $\alpha pn\in {\rm
spec}(A_{n,\kappa,p})$. Assertion $3$ follows from
\lemref{eigenvalues} applied to $T_{\sG_{n,p,\kappa}}$ and
$T_\kappa$.

As for Assertion $4$, we recall from Claim \ref{claim:puf} that
whenever $\beta pn\in{\rm spec}(A_{n,p,\kappa})$ with corresponding
eigenspace projection $\Pi_\beta$ the corresponding eigenspace of
$T_{\sG_{n,\kappa,p}}$ is $H_n\Pi_\beta E_n$. This implies that:
$$H_n \Pi_{(\alpha-\gamma)pn,(\alpha+\gamma)pn}(A_{n,p,\kappa}) E_n$$
is the projection onto the eigenspaces of $T_{\sG_{n,p,\kappa}}$ corresponding to eigenvalues between $\alpha-\gamma$ and $\alpha+\gamma$. One can apply \lemref{perturbation} with $\eps=\theta$ and $b-\gamma=a+\gamma=\alpha$ to deduce that, whenever $\alpha$ is as in assertion $4$ and $\|T_{\sG_{n,p,\kappa}}-T_\kappa\|\leq \theta$,
$$\|H_n \Pi_{(\alpha-\gamma)pn,(\alpha+\gamma)pn}(A_{n,p,\kappa}) E_n - P_\alpha\|_{L^2\to L^2}\leq \frac{4\theta}{\pi(\gamma-\theta)}.$$
Multiplying both operators above by $E_n$ on the left and by $H_n$
on the right, using that $H_n$ and $E_n$ have norm $\leq 1$ and that
$E_nH_n=I_n$, we see that:
$$\|\Pi_{(\alpha-\gamma)pn,(\alpha+\gamma)pn}(A_{n,p,\kappa}) - E_n P_\alpha H_n\|\leq \frac{4\theta}{\pi(\gamma-\theta)}.$$

This finishes the proof modulo inequality \eqnref{goooooo}, which is
the subject of \lemref{approx} below.

\subsubsection{Approximating $T_\kappa$}
\begin{lemma}\label{lem:approx}Under \assumref{kappa}, suppose $\eps>0$ is given and $\kappa_\eps:[0,1]^2\to \R_+\cup\{0\}$ is a $L$-Lipschitz symmetric function, with values between $0$ and $K$, such that
$$\int_{0}^1\int_{0}^1(\kappa(x,y)-\kappa_\eps(x,y))^2\,dx\,dy\leq \eps^2.$$
Then the following holds with probability $\geq 1/2n^2$:
$$\|T_\kappa - \overline{T}\|_{L^2\to L^2}\leq 2\eps + c(L+K)\,\left(\frac{\ln n}{n}\right)^{1/4},$$
where $c>0$ is universal.\end{lemma}
\begin{proof}Define:
$$\widehat{T}\equiv \sum_{1\leq i,j\leq n}\kappa_\eps(X_i,X_j)\Ind{\left(\frac{\sigma_n(i)-1}{n},\frac{\sigma_n(i)}{n}\right]\times \left(\frac{\sigma_n(j)-1}{n},\frac{\sigma_n(j)}{n}\right]}.$$
We will bound:
\begin{equation}\label{eq:threeparts}\|T_\kappa - \overline{T}\|_{L^2\to L^2}\leq \|T_\kappa - T_{\kappa_\eps}\|_{L^2\to L^2} + \|\overline{T}-\widehat{T}\|_{L^2\to L^2} + \|T_{\kappa_\eps}-\widehat{T}\|_{L^2\to L^2}.\end{equation}

By the results in \secref{integralop}, one can bound the first term in the RHS by:
$$\|T_\kappa - T_{\kappa_\eps}\|^2_{L^2\to L^2} = \|T_{\kappa-\kappa_\eps}\|^2_{L^2\to L^2}\leq \int_{0}^1\int_{0}^1(\kappa(x,y)-\kappa_\eps(x,y))^2\,dx\,dy\leq \eps^2.$$
For the second term, we observe that $\overline{T}-\widehat{T}$ is
of the form $T_\eta$ for $\eta$ taking the values
$\kappa(X_i,X_j)-\kappa_\eps(X_i,X_j)$ on squares of area $1/n^2$.
We deduce from the results in \secref{integralop} that:
\begin{equation}\label{eq:secondterm}\|\overline{T} - \widehat{T}\|^2_{L^2\to L^2}\leq \frac{1}{n^2}\sum_{i,j=1}^n(\kappa(X_i,X_j)-\kappa_\eps(X_i,X_j))^2.\end{equation}
The expected value of the RHS is:
$$\int_{0}^1\int_{0}^1(\kappa(x,y)-\kappa_\eps(x,y))^2\,dx\,dy\leq \eps^2$$
Moreover, the random variables $X_i$ are independent and replacing
$X_i$ by some other $X_i'\in[0,1]$ can change the value of the sum
in the RHS of \eqnref{secondterm} by at most $K^2/n$ (as each term
is bounded by $K$ and only $n$ terms involve $X_i$). Azuma's
inequality \cite{AlonSpencer_ProbMethod} implies:
$$\Pr{\frac{1}{n^2}\sum_{i,j=1}^n(\kappa(X_i,X_j)-\kappa_\eps(X_i,X_j))^2\geq \eps + t}\leq e^{-nt^2/2K^4}.$$
Therefore, with probability $\geq 1-1/4n^2$ we have:
$$\|\overline{T} - \widehat{T}\|_{L^2\to L^2}\leq \sqrt{\eps^2 + K^2\sqrt{\frac{2\ln(4n^2)}{n}}}\leq \eps + c\,K\left(\frac{\ln n}{n}\right)^{1/4},$$
where $c>0$ is some universal constant. We deduce:
\begin{equation}\label{eq:firsttwoterms}\|T_\kappa - T_{\kappa_\eps}\|_{L^2\to L^2} + \|\overline{T}-\widehat{T}\|_{L^2\to L^2}\leq 2\eps + c\,K\left(\frac{\ln n}{n}\right)^{1/4}\end{equation}
with probability $\geq 1-1/4n^2$.

To finish the proof, we must bound the third term in \eqnref{threeparts}. To do this, we notice that:
$$\widehat{T} -T_{\kappa_\eps}=T_\eta$$
where
$$\eta\equiv \sum_{1\leq i,j\leq n}\left(\kappa_\eps(X_i,X_j) - \kappa_\eps\right)\Ind{\left(\frac{\sigma_n(i)-1}{n},\frac{\sigma_n(i)}{n}\right]\times \left(\frac{\sigma_n(j)-1}{n},\frac{\sigma_n(j)}{n}\right]}.$$
Using the definition of $\sigma_n$ from \secref{preciseresult}, one
can rewrite this as:
$$\eta(x,y)=\sum_{1\leq i,j\leq n}\left(\kappa_\eps(\overline{X}_i,\overline{X}_j) - \kappa_\eps(x,y)\right)\Ind{\left(\frac{i-1}{n},\frac{i}{n}\right]\times \left(\frac{j-1}{n},\frac{j}{n}\right]}(x,y).$$
Recall that $\kappa_\eps$ is $L_\eps$-Lipschitz and therefore,
\begin{multline*}\forall (x,y)\in \left(\frac{i-1}{n},\frac{i}{n}\right]\times \left(\frac{j-1}{n},\frac{j}{n}\right],
|\kappa_\eps(\overline{X}_i,\overline{X}_j) - \kappa_\eps(x,y)|\leq
\\ \leq 2L_\eps/n + |\kappa_\eps(\overline{X}_i,\overline{X}_j) -
\kappa_\eps(i/n,j/n)|\leq \\ \leq 2L_\eps/n +
L_\eps|\overline{X}_i-i/n| +
L_\eps|\overline{X}_j-j/n|.\end{multline*} Integrating $\eta^2$, we
find that:
\begin{multline*}\int_{[0,1]^2}\eta^2\leq \frac{1}{n^2}\sum_{i,j=1}^n (2L_\eps/n + L_\eps|\overline{X}_i-i/n| + L_\eps|\overline{X}_j-j/n|)^2\leq \\ \mbox{[use $(a+b+c^2\leq 3(a^2 +b^2+c^2))$]} \leq \frac{12L_\eps^2}{n^2} + 6L_\eps^2\max_{1\leq i\leq n}(\overline{X}_i-i/n)^2.\end{multline*}
A simple calculation using e.g. Massart's version of the
Dvoretsky-Kiefer-Wolfowitz inequality
\cite{Massart_DvoretskyKieferWolfowitz} reveals that the last term
is $\leq c^2\ln n/n$ ($c>0$ universal) with probability $\geq
1-1/4n^2$. We deduce that:
$$\|\widehat{T} -T_{\kappa_\eps}\|_{L^2\to L^2}=\|T_\eta\|_{L^2\to L^2}\leq \sqrt{\int_{[0,1]^2}\eta^2}\leq \frac{2\sqrt{3}L}{n} + cL\sqrt{\frac{6\ln n}{n}}$$
with probability $\geq 1-1/4n^2$. Combining this with \eqnref{firsttwoterms} and replacing $c>0$ with a larger universal constant if necessary finishes the proof.\end{proof}

\section{Freedman's inequality for matrix
martingales}\label{sec:freedman}

In this Section we prove our new concentration inequality,
\thmref{freedman}. We begin with some preliminaries from matrix
analysis.

\subsection{Preliminaries from matrix analysis}\label{sec:prelim}

\subsubsection{The positive semi-definite order}

Matrix inequalities for the positive semi-definite order will be
essential in our proof.

Given $A\in \C^{d\times d}_{\rm Herm}$, say that $A\succeq 0$ if $A$
is positive semi-definite, which is the same as saying that all
eigenvalues of $A$ are non-negative, or that $v^*Av\geq 0$ for all
$v\in\C^d$. We will also write $A\preceq B$ (for $B\in\C^{d\times
d}_{\rm Herm}$) if $B-A\succeq 0$. Notice that $A\preceq \xi I$ for
some $\xi\in\R$ iff $\lambda_{\max}(A)\leq \xi$.

We will need four other properties of the partial order
``$\preceq$". The first three are easily checked and we omit their
proofs:
\begin{equation}\label{eq:closed}\mbox{The set }\{(A,B)\in \left(\C^{d\times d}_{\rm Herm}\right)^2\,:\, A\preceq B\}\mbox{ is closed in the product topology.}\end{equation}
\begin{equation}\label{eq:linearorder}\forall \{A_i\}_{i=1}^k, \{B_i\}_{i=1}^k\subset \C^{d\times d}_{\rm Herm}:\, ``\forall 1\leq i\leq k,\, A_i\preceq B_i"\Rightarrow ``\sum_{i=1}^kA_i \preceq \sum_{i=1}^k B_i".\end{equation}
\begin{equation}\label{eq:firsteigen}\forall A,B\in \C^{d\times d}_{\rm Herm}:\, ``A\succeq 0"\Rightarrow ``\lambda_{\max}(A+B)\geq \lambda_{\max}(B)".\end{equation}

The fourth one is slightly less standard.
\begin{equation}\label{eq:increasing}\forall A,B,C\in \C^{d\times d}_{\rm Herm},\, (A\succeq 0\wedge  C-B\succeq 0)\Rightarrow \tr(AB)\leq \tr(AC).\end{equation}

To prove \eqnref{increasing}, notice that for for $A,B,C$ as above,
$$\tr(A(C-B)) = \tr((C-B)^{1/2}A(C-B)^{1/2})$$ where $(C-B)^{1/2}\in\C^{d\times d}_{\rm Herm}$
is the (also positive semi-definite) square root of $C-B$. Then
notice that for any $v\in \C^n$,
$$v^*(C-B)^{1/2}A(C-B)^{1/2}v = [(C-B)^{1/2}v]^*A[(C-B)^{1/2}v]\geq
0$$ since $A\succeq 0$. This implies that $(C-B)^{1/2}A(C-B)^{1/2}$
must be positive semi-definite, hence its trace is non-negative:
$\tr(A(C-B))\geq 0$, which is equivalent to \eqnref{increasing} by
linearity.
\subsubsection{Conditional expectations are monotone}
We will also need the following property that relates expectations to the positive semi-definite order. Let $X,Y$ be integrable, random $d\times d$ Hermitian matrices defined on a common probability space $(\Omega,\sF,\Prwo)$. Then:
\begin{equation}\label{eq:exisincreasing}\mbox{If }X\preceq Y\mbox{ almost surely, then }\Ex{X\mid\sG}\preceq\Ex{Y\mid \sG}\mbox{ almost surely}.\end{equation}
To see this, it suffices to see that for all $v\in\C^d$, $v^*Xv\leq  v^*Yv$ and therefore $\Ex{v^*Xv\mid\sG}\leq \Ex{v^*Yv\mid \sG}$. However, our definition of $\Ex{\cdot\mid\sG}$ for matrices (cf. \secref{prelimprob}) implies that $\Ex{v^*Xv\mid\sG}=v^*\Ex{X\mid\sG}v$ and $\Ex{v^*Yv\mid\sG}=v^*\Ex{Y\mid\sG}v$. Therefore, if
$$X\preceq Y\mbox{ almost surely }\Rightarrow \forall v\in\C^d,\, ``v^*\Ex{X\mid\sG}v\leq v^*\Ex{Y\mid\sG}v \mbox{ almost surely}".$$
Now let $Q\subset \C^d$ be dense and countable. Note that for all $A\in\C^{d\times d}_{\rm Herm}$, $A\preceq 0$ if and only if $v^*Av\geq 0$ for all $v\in Q$.
$$\Ex{X\mid\sG}\preceq \Ex{Y\mid\sG}\mbox{ a.s.}\Leftrightarrow \Pr{\forall v\in Q,\, v^*\Ex{X\mid\sG}v\leq v^*\Ex{Y\mid\sG}v}=1$$
and the RHS follows from $X\preceq Y$ by the previous implication (since $Q$ is countable).

\subsubsection{Matrix functions and matrix exponentials}

If $f:\C\to\C$ given by a power series $f(x) =
\sum_{i=1}^{\infty}c_ix^i$ that converges for all $x\in\C$, one may
define:
$$f(A)\equiv \sum_{i=1}^{\infty} c_i A^i, A\in\C^{d\times d},$$
which can be shown to converge for all $A$. $f(A)$ is Hermitian
whenever $A\in\C^{d\times d}_{\rm Herm}$ and the coefficients $c_i$
belong to $\R$. In that case, the eigenvalues of $f(A)$ are given by
$f(\lambda_i(A))$ for $0\leq i\leq d-1$, with the same eigenvectors
as $A$. In particular, $f(A)\preceq \xi I$ for some $\xi\in\R$ iff
$f(\lambda_i(A))\leq \xi$ for each $0\leq i\leq d-1$. Moreover, for
all $s\geq 0$,
\begin{equation}\label{eq:traceandbound}\exp(s\lambda_{\max}(A)) =
\lambda_{\max}(\exp(sA))\leq \tr(\exp(sA)).\end{equation}

We need one more result from matrix analysis, called the {\em Golden
Thompson inequality}.

\begin{equation}\label{eq:goldenthompson}\forall d\in\{1,2,3,\dots\},\,\forall A,B\in\C^{d\times d}_{\rm Herm}:\, \tr(e^{A+B})\leq \tr(e^{A}e^{B}).\end{equation}

This inequality is fundamental in adapting the standard proofs of
concentration to the matrix setting
\cite{AhlswedeWinter_StrongConverse,ChristofidesMarkstrom_HoeffdingForMatrices,Eu_Rudelson}.

\subsection{The proof}

We begin with two simple Lemmas.

\begin{lemma}For any matrix $C\in\C^{d\times d}_{\rm Herm}$ and $k\in\N\backslash\{0,1\}$, $C^k\preceq \|C\|_2^{k-2} C^2$.\end{lemma}
\begin{proof}$\|C\|_2^{k-2} C^2 - C^k$ has the same eigenvectors as $C$ and its eigenvalues are given by $$\|C\|_2^{k-2}\lambda_i(C)^2 - \lambda_i(C)^{k}= (\|C\|_2^{k-2} - \lambda_i(C)^{k-2})\lambda_i(C)^2.$$ This is always $\geq 0$ because $\|C\|_2=\max_{1\leq i\leq d}|\lambda_i(C)|$.\end{proof}

\begin{lemma}\label{lem:boundexponential}For any matrix $C\in\C^{d\times d}_{\rm Herm}$ with $\|C\|_2\leq 1$, $e^{C}\preceq I + C + C^2$.\end{lemma}
\begin{proof}The previous lemma implies that $C^i\preceq
C^2$ for all $i\geq 2$. Property \eqnref{linearorder} of
``$\preceq$" implies that for any $k$,
$$I + C + \sum_{i=2}^{k} \frac{C^i}{i!}\preceq I + C +
\left(\sum_{i=2}^k\frac{1}{i!}\right)C^2\preceq I + C + C^2.$$ Now
let $k\nearrow +\infty$ and use \eqnref{closed}.\end{proof}

The next step is an exponential inequality for martingales.

\begin{lemma}[Exponential inequality for martingales]\label{lem:freedman}Let
$Z_n$, $W_n$ be as in \thmref{freedman} with $M=1$. Then for all
$s\in [0,1/2]$ and all deterministic $C\in\C^{d\times d}_{\rm
Herm}$,
$$\Ex{\tr\left[\exp\left(sZ_n  - 2s^2 W_n+ C\right)\right]}\leq\tr\left[\exp\left(C\right)\right].$$\end{lemma}

\begin{proof}Set $X_n\equiv Z_n-Z_{n-1}$ and $\Delta_n\equiv \Ex{X_n^2\mid\sF_{n-1}}$. We use Golden Thompson \eqnref{goldenthompson} to deduce that:
$$\tr(e^{sZ_n - 2s^2 W_n + C})\leq \tr(e^{sX_n - 2s^2 \Delta_n}e^{sZ_{n-1} - 2s^2 W_{n-1} + C}).$$
Taking conditional expectations, we see that:
\begin{eqnarray*}\Ex{\tr(e^{sZ_n - s^2 W_n + C})\mid\sF_{n-1}}&\leq&
\Ex{\tr(e^{sDX_n - 2s^2\Delta_n}e^{sZ_{n-1} - 2s^2 W_{n-1} +
C})\mid\sF_{n-1}} \\ &=& \tr(\Ex{e^{sX_n -
2s^2\Delta_n}\mid\sF_{n-1}}\,e^{sZ_{n-1} - 2s^2 W_{n-1} +
C}).\end{eqnarray*} Here the equality is a result of $\tr$ and
expected values commuting \eqnref{traceexp}, as well as noting that
$e^{sX_{n-1} - 2s^2 W_{n-1} + C}$ is $\sF_{n-1}$-measurable and then
applying \eqnref{product} to the conditional expectation.

We now make the following claim.\begin{claim}$\Ex{e^{sX_n -
2s^2\Delta_n}\mid\sF_{n-1}}\preceq I$.\end{claim} This will imply
(via monotonicity of the trace \eqnref{increasing}) that:
$$\Ex{\tr(e^{sZ_n - 2s^2 W_n + C})\mid \sF_{n-1}} \leq \tr(e^{sZ_{n-1} - 2s^2 W_{n-1} +
C}),$$ hence
$$\Ex{\tr(e^{sZ_n - 2s^2 W_n + C})} \leq \Ex{\tr(e^{sZ_{n-1} - 2s^2 W_{n-1} +
C})}$$ and the Lemma follows from this via induction in $n$.

To prove the claim, we first note that for $|s|\leq 1/2$,
$$\|sX_n - 2s^2\Delta_n\|_2\leq \frac{\|X_n\|_2 + \|\Delta_n\|_2}{2}\leq 1$$
by the assumption that $\|X_n\|_2\leq 1$. We now apply
\lemref{boundexponential} with $C = sX_n - s^2\Delta_n$ and the monotonicity of conditional expectations \eqnref{increasing} to obtain:
$$\Ex{e^{sX_n -
2s^2\Delta_n}\mid\sF_{n-1}} \preceq \Ex{I + sX_n - 2s^2 \Delta_n +
s^2 X_n^2 -2s^3 X_n\Delta_n - 2s^3X_n\Delta_n+ 4s^4\Delta_n^2\mid
\sF_{n-1}}.$$

$\Delta_n=\Ex{X_n^2\mid\sF_{n-1}}$ is $\sF_{n-1}$-measurable and the
martingale property implies $\Ex{X_n\mid\sF_{n-1}}=0$. Via equation
\eqnref{product}, this implies $\Ex{\Delta_nX_n\mid
\sF_{n-1}}=\Ex{X_n\Delta_n\mid \sF_{n-1}}=0$ almost surely. This
means that the RHS above is a.s. equal to:
$$I - s^2 \Delta_n + 4s^4\Delta_n^2.$$
Now notice that the eigenvalues of $- s^2 \Delta_n + 4s^4\Delta_n^2$
are given by:
$$- s^2 \lambda_i(\Delta_n) + 4s^4\lambda_i(\Delta_n)^2, 1\leq i\leq d.$$
The inequality $s\leq 1/2$ implies $4s^4\leq s^2$. Moreover, each
$\lambda_i(\Delta_n)$ is between $0$ and $1$, since
$\|\Delta_n\|\leq 1$ and $\Delta_n\succeq 0$ (it is the conditional
expectation of $X_n^2$). This implies that the above expression is
at most:
$$- s^2 \lambda_i(\Delta_n) + s^2\lambda_i(\Delta_n)=0$$
for each $i$. Therefore, $- s^2 \Delta_n + 4s^4\Delta_n^2\preceq 0$
and (again using the monotonicity property \eqnref{exisincreasing}),
$\Ex{e^{sX_n - 2s^2\Delta_n}\mid\sF_{n-1}} \preceq I$ almost
surely.\end{proof}

\begin{proof}[of \thmref{freedman}] One may assume that $M=1$ (one can always rescale $Z_n$ so that this is the case; the bound behaves accordingly). If $\lambda_{\max}(W_n)\leq \sigma^2$,
$\sigma^2 I - W_n\succeq 0$ is positive semi-definite. Inequality
\eqnref{firsteigen} then implies that for all $s>0$,
$$\lambda_{\max}(sX_n + 2s^2\sigma^2 I - 2s^2W_n)\geq \lambda_{\max}(sX_n) =
s\lambda_{\max}(X_n).$$ Therefore, \begin{eqnarray*}\forall s>0,\,
\Pr{\lambda_{\max}(X_n)\geq t,\, \lambda_{\max}(W_n)\leq \sigma^2} &\leq&
\Pr{\lambda_{\max}(sX_n + 2s^2\sigma^2I - 2s^2W_n)\geq st} \\ &\leq&
e^{-st}\Ex{\exp(\lambda_{\max}(sX_n + 2s^2\sigma^2I -
2s^2W_n))}.\end{eqnarray*} We now use the inequality
``$e^{\lambda_{\max}(sZ)}\leq \tr(e^{sZ})$", valid for any $s\geq 0$ and
$Z\in\C^{d\times d}_{\rm Herm}$ (cf. \eqnref{traceandbound}),
together with the exponential inequality in \lemref{freedman}, to
deduce that for all $s\in [0,1/2]$,
\begin{eqnarray*}\Pr{\lambda_{\max}(sX_n + 2s^2\sigma^2I - 2s^2W_n)\geq
st} &\leq& e^{-st}\Ex{\tr(\exp(sX_n + 2s^2\sigma^2I -
2s^2W_n))}\\ &\leq& \tr(\exp(2s^2 \sigma^2 I))e^{-st} = d\,
e^{2s^2\sigma^2 - st}.\end{eqnarray*} Set
$$s\equiv \frac{t}{4\sigma^2 + 2t}.$$
Notice that with this choice $s\leq 1/2$ always. Moreover,
$$2s^2\sigma^2 = \frac{t^2}{8\sigma^2(1+t/2\sigma^2)^2}\leq
\frac{t^2}{8\sigma^2(1+t/2\sigma^2)} = \frac{st}{2}.$$ Hence:
$$\Pr{\lambda_{\max}(X_n)\geq t,\, \lambda_{\max}(W_n)\leq \sigma^2} \leq d\,e^{-\frac{t^2}{8\sigma^2 +
4t}},$$ as desired.\end{proof}

\begin{remark}\label{rem:indepcase}It is well-known in the scalar case that inequalities for martingales imply inequalities for independent sums. The same is true in the matrix setting. Let $X_1,\dots,X_n$ be mean-zero independent random matrices, defined on a common probability space $(\Omega,\sF,\Prwo)$, with values in $\C^{d\times d}_{\rm Herm}$ and such that there exists a $M>0$ with $\|X_i\|\leq M$ almost surely for all $1\leq i\leq m$. Letting $\sF_0=\{\emptyset,\Omega\}$ and $\sF_i=\sigma(X_1,\dots,X_i)$ ($i\in[n]$), one can see that:
$$\{(Z_i\equiv \sum_{j=1}^iX_j, \sF_i)\}_{i=0}^n$$
is a martingale satisfying the assumptions of the Theorem and that, moreover, $W_n$ is deterministic in this case:
$$W_n\equiv \sum_{i=1}^n \Ex{(Z_i-Z_{i-1})^2\mid\sF_{i-1}} = \sum_{i=1}^n\Ex{X_i^2}.$$
Thus one may take:
$$\sigma^2 = \lambda_{\max}\left(\sum_{i=1}^n\Ex{X_i^2}\right)$$
in \thmref{freedman} and deduce the first half of the Corollary
below. The other half comes from considering
$-\sum_{i=1}^nX_i$.\end{remark}

\begin{corollary}\label{cor:freedman}Let $X_1,\dots,X_n$ be mean-zero independent random matrices, defined on a common probability space $(\Omega,\sF,\Prwo)$, with values in $\C^{d\times d}_{\rm Herm}$ and such that there exists a $M>0$ with $\|X_i\|\leq M$ almost surely for all $1\leq i\leq m$. Define:
$$\sigma^2\equiv \lambda_{\max}\left(\sum_{i=1}^n\Ex{X_i^2}\right).$$
Then for all $t\geq 0$,
$$\Pr{\lambda_{\max}\left(\sum_{i=1}^n X_i\right)\geq t}\leq d\, e^{-\frac{t^2}{8\sigma^2 + 4Mt}},$$
and
$$\Pr{\left\|\sum_{i=1}^n X_i\right\|\geq t}\leq 2d\, e^{-\frac{t^2}{8\sigma^2 + 4Mt}}.$$\end{corollary}

\section{Final remarks}\label{sec:final}

\noindent {\em Sharpness of \thmref{freedman}.} One can show that
\thmref{freedman} is close to sharp and that, in particular, the $d$
factor in the bound is necessary for general martingale sequences.
To see this, consider a sum $Z_n$ of $n$ independent, identically
distributed $d\times d$ diagonal random matrices $X_1,\dots,X_n$
whose diagonal entries are independent, unbiased $\pm 1$. The
largest eigenvalue of $Z_n$ is a maximum of $d$ independent random
sums, each with $n$ terms of the kind $\pm 1$ above. One can see
that for large $n$ and $d$ and for $t\approx \sqrt{n\ln d}$,
$$\Pr{\lambda_{\max}(Z_n)\geq t}\geq d e^{-(1+\liloh{1})t^2/2n}$$
which is what \corref{freedman} gives up to the constants in the
exponent.

An interesting question is to understand the circumstances under which one can remove the $d$ factor from the bound. For instance, can the sharper results of \cite{FurediKomlos_Eigenvalues,FeigeOfek_Spectral} be reobtained via some variant of \thmref{freedman}?\\

\noindent {\em Other applications of \thmref{freedman}.} In a
related paper (in preparation) we show how \thmref{freedman} can be
used to show concentration of the matrices of random lifts of large
graphs. A pleasing corollary of our result is this: consider a
random $k_1k_2$-lift of a large graph $G$ with minimum degree
$\omega(\ln (k_1k_2n))$. The Laplacian of this lift is essentially
indistinguishable from that of the (in principle very different)
random graph obtained by performing a $k_1$-lift on $G$ and then a
$k_2$-lift on the resulting graph.

It would be interesting to see other applications of
\thmref{freedman}, especially in settings where the
Christofides-M\"{a}rkstrom bound is useless because its variance
term is too large (cf. \remref{comparisons}).\\

\noindent {\em The Laplacian of inhomogeneous random graphs.} The
results of the \secref{inhomogeneous} can be extended to the
Laplacian $\sL_{n,p,\kappa}$ of $G_{n,p,\kappa}$. More precisely,
add the following condition to \assumref{kappa}: that there exists a
$K_->0$ such that for all $x\in[0,1]$, $\kappa(x)\equiv
\int_0^1\kappa(x,y)\,dy\geq K_-$. Then there is a close
correspondence between $\sL_{n,p,\kappa}$ and the operator
$S_\xi\equiv {\rm Id}_{L^2} - T_{\xi}$, where ${\rm Id}_{L^2}$ is
the identity operator on $L^2([0,1])$ and $T_\xi$ is the integral
operator given by the symmetric, non-negative function:
$$\xi(\cdot,\cdot\cdot) =
\frac{\kappa(\cdot,\cdot\cdot)}{\sqrt{\kappa(\cdot)\kappa(\cdot\cdot)}}.$$

That is, if $p\leq 1/K$ and $pnK_-\gg C\ln n$ for some $C$, we will
have:
$$\|\sL_{n,p,\kappa}-E_nS_{\xi}H_n\|=\liloh{1} \mbox{ and
}\|H_n\sL_{n,p,\kappa}E_n-S_{\xi}\|=\liloh{1},$$ with consequences
for the spectrum and eigenspaces of $\sL_{n,p,\kappa}$. We omit the
details.\\

\noindent {\em Better bounds and extensions?} We have mentioned the
results on spectral gaps in references \cite{FeigeOfek_Spectral} and
\cite{CojaLanka_SpectrumGivenDegrees}, on $G_{n,p}$ and random
graphs with given expected degrees. These papers actually do much
more than we described, as they show that, even is very sparse
graphs, there is a large ``core" set of vertices so that the
matrices of the induced subgraph are well-behaved. It would be an
interesting question to prove a similar result either for more
general instances of bond percolation or inhomonegeous random
graphs.\\

\noindent {\em Cut convergence, eigenvalues and eigenvectors.} It is
not clear to the author what one can/cannot prove about eigenvectors
and eigenvalues of sparse graphs while only assuming that they
converge to a given $\kappa$ in the cut norm. Ideally, one would wish to be able to prove that this suffices for the convergence of the given operators, at least under suitable assumptions, but it is not clear how one should proceed.

\appendix \section{Appendix: two perturbation results}
The following functional-analytic perturbation results are needed in
the main text. In what follows $\sH$ is a real Hilbert
space and $\|\cdot\|$ denotes both the Hilbert space norm and the
induced norm on linear operators. Undefined notions and quoted
results can be found in any textbook on Functional Analysis, eg.
\cite{RieszSzNagy_FunctionalAnalysis,Rudin_FunctionalAnalysis}.
\begin{lemma}\label{lem:eigenvalues}Suppose $V,W$ are compact Hermitian linear operators on the Hilbert space $\sH$ that satisfy $\|V-W\|\leq \eps$. Let ${\rm spec}(V)$, ${\rm spec(W)}$ denote the spectra of $V$ and $W$ (respectively). Let $S\subset\R$ be such that $\inf_{s\in S}|s|>\eps$ and let $m_V(S)$ be the sum of the multiplicities of all elements of ${\rm spec}(V)\cap S$. Then:
$$m_V(S)\leq m_W(S^\eps)$$
where for $A\subset \R$, $A^{\eps}\equiv \{x\in \R\,:\, \exists a\in A,\, |x-a|\leq \eps\}$.\end{lemma}

\begin{proof}This is evident if both $V$ and $W$ have finite-dimensional rank. In this case one may restrict to the span of the two ranges, which is a finite-dimensional space isomorphic to some $\R^d$, and then apply \eqnref{eigenperturbation}. [Do notice that $0$ might belong to the spectrum of the restriction of $V$ or $W$ to the finite-dimensional subspace, even though it does not belong to the original spectra. This, however, will not matter, due to the condition $\inf_{s\in S}|s|>\eps$.]

For the case of infinite-dimensional rank, $V$ and $W$ are the limit
(in the operator norm) of operators of finite-dimensional rank. More
specifically, recall from \secref{integralop} that the spectral theorem for compact, self-adjoint operators states that $V$ can be written as a sum:
$$V = \sum_{\alpha\in{\rm spec}(V)}\alpha P_\alpha$$
where the $P_{\alpha}$ are orthogonal projectors of orthogonal
ranges, with finite rank if $\alpha\neq 0$. Moreover, for any
$\delta>0$, ${\rm spec}(V)\backslash(-\delta,\delta)$ is finite.
Therefore, the finite-rank operator:
$$V_\delta = \sum_{\alpha\in{\rm spec}(V)\backslash(-\delta,\delta)}\alpha P_\alpha$$
satisfies $\|V_\delta-V\|\leq \delta$. One may similarly define  $W_\delta$ with $\|W_\delta-W\|\leq \delta$ and it follows that $\|V_\delta-W_\delta\|\leq \eps+2\delta$. Moreover, we have the simple fact:
\begin{equation}\label{eq:A}\forall A\subset \R\backslash[-\delta,\delta],\, m_{V_\delta}(A)=m_{V}(A)\mbox{ and }m_{W_\delta}(A)=m_W(A).\end{equation}

Let $\delta>0$ be small, so that $\inf_{s\in S}|s|>\eps+3\delta$.  The finite-dimensional result implies:
$$m_{V_\delta}(S)\leq m_{W_\delta}(S^{\eps+2\delta}).$$
Notice that $m_{V_\delta}(S)=m_V(S)$ because $S\subset \R\backslash[-\eps,\eps]\subset \R\backslash[-\delta,\delta]$ and therefore \eqnref{A} applies. Moreover, $\forall x\in S^{\eps+2\delta}$,
$$|x|\geq \inf_{s\in S} |s| - \eps - 2\delta >\delta$$
by the choice of $\delta$; therefore $S^{\eps+2\delta}\subset \R\backslash[-\delta,\delta]$ and we can apply \eqnref{A} again to deduce that $m_{W_\delta}(S^{\eps+2\delta})=m_{W}(S^{\eps+2\delta})$. These facts imply:
$$m_{V}(S)\leq m_{W}(S^{\eps+2\delta}).$$
It is an exercise to show that $m_{W}(S^{\eps+2\delta})\to m_W(S^\eps)$ when $\delta\searrow 0$. This finishes the proof.\end{proof}

\begin{lemma}\label{lem:perturbation}Suppose $V,W$ are compact Hermitian linear operators on the Hilbert space $\sH$ that satisfy $\|V-W\|\leq \eps$. Assume that $a<b$and $\gamma>\eps$ be such that
$a+\gamma<b-\gamma$ and $V$ does not contain any eigenvalues in
$(a-\gamma,a+\gamma)\cup (b-\gamma,b+\gamma)$. Define $\Pi_{a,b}(V)$
as the projector onto the span of the eigenvectors of $V$
corresponding to $a\leq \lambda_k(V)\leq b$ and define
$\Pi_{a,b}(W)$ similarly. Then:
$$\|\Pi_{a,b}(V) - \Pi_{a,b}(W)\|\leq \frac{(b-a +2\gamma)\,\eps}{\pi(\gamma^2-\gamma\eps)}.$$\end{lemma}

\begin{proof}Suppose first that $\sH$ is
finite-dimensional, in which case one may assume that $\sH=\C^d$ for
some $d$ and that $V$ and $W$ are matrices. In this case we use a
standard technique involving contour integration in the complex
plane and the resolvent of linear operators \cite[Chapter
2]{Kato_PerturbationTheory}.

Let $\sC$ be the rectangular contour in the complex plane that
passes through the points $a+\gamma \sqrt{-1}$, $a-\gamma
\sqrt{-1}$, $b-\gamma \sqrt{-1}$, $b+\gamma \sqrt{-1}$ in
counterclockwise order. The Cauchy formula implies that for all
$\lambda\in\R\backslash\{a,b\}$,
$$\frac{1}{2\pi\sqrt{-1}}\int_{\sC} \frac{dz}{z-\lambda} = \left\{\begin{array}{ll} 1, & a<\lambda<b\\ 0,&\mbox{otherwise.}\end{array}\right.$$
Now consider the {\em resolvent}:
$$R_V(z)\equiv (zI-V)^{-1},\, z\in\C\backslash\{\lambda_i(V)\,:\,0\leq i\leq n-1\}.$$
The spectral theorem implies that:
$$R_V(z) = \sum_{k=0}^{d-1}\,\frac{\psi_{k,V}\psi_{k,V}^*}{z-\lambda_k(V)}.$$
where $\psi_{k,V}$ is the eigenvector of $V$ corresponding to
$\lambda_k(V)$. By assumption, $V$ has no eigenvalues on $\sC$,
therefore:
$$\frac{1}{2\pi\sqrt{-1}}\int_\sC R_V(z)\,dz = \sum_{k=0}^{d-1} \frac{1}{2\pi\sqrt{-1}}\int_\sC \frac{\psi_{k,V}\psi_{k,V}^*}{z-\lambda_k(V)}\,dz = \sum_{k:\lambda_k(V)\in [a,b]} \psi_{k,V}\psi_{k,V}^* = \Pi_{a,b}(V).$$
Now define the resolvent $R_W(z) = (zI - W)^{-1}$. Recall that
$|\lambda_i(V)-\lambda_i(W)|\leq \eps<\gamma$ by
\eqnref{eigenperturbation} and that no eigenvalue of $V$ lies in
$(a-\gamma,a+\gamma)\cup (b-\gamma,b+\gamma)$ (by assumption). This
implies that no eigenvalue of $W$ can lie on $a$ or $b$. Therefore,
the same reasoning used above implies that:
$$\frac{1}{2\pi\sqrt{-1}}\int_\sC R_W(z)\,dz = \Pi_{a,b}(W).$$
In particular,
$$\|\Pi_{a,b}(V) - \Pi_{a,b}(W)\| = \left\|\frac{1}{2\pi\sqrt{-1}}\int_\sC (R_V(z)-R_W(z))\,dz\right\|.$$
It is not hard to show that:
$$\left\|\frac{1}{2\pi\sqrt{-1}}\int_\sC (R_V(z)-R_W(z))\,dz\right\|\leq \frac{1}{2\pi}\int_\sC \|R_V(z)-R_W(z)\|\,d|z|.$$
Since $\sC$ has length $2(b-a)+4\gamma$, we have:
\begin{equation}\label{eq:differenceresolvent}\|\Pi_{a,b}(V) - \Pi_{a,b}(W)\| \leq \frac{(b-a
+2\gamma)}{\pi}\,\max_{z\in \sC}\|R_V(z) - R_W(z)\|.\end{equation} We
now bound the difference between the resolvents. Recall that for
$T\in\C^{d\times d}$ with $\|T\|<1$,
$$(I+T)^{-1} = \sum_{n\geq 0}T^n.$$
Suppose we can show that $\|(W-V)R_V(z)\|\leq \alpha<1$ for $z\in
\sC$. Then:
\begin{eqnarray*}\|R_W(z) -R_V(z)\| &=& \|((zI - V) - (W-V))^{-1} - R_V(z)\|\\ &=&\|(zI - V)^{-1}\,(I - (W-V)(zI -
V)^{-1})^{-1}-R_V(z)\|\\
&=& \|R_V(z)\,\{(I - (W-V)R_V(z))^{-1}-I\}\|\\ &=& \|\sum_{n\geq 1}
R_V(z)[(W-V)\,R_V(z)]^{n}\| \\&\leq &\|R_V(z)\| \sum_{n\geq 1}
\|(W-V)\,R_V(z)\|^{n}\\ & \leq &
\|R_V(z)\|\,\frac{\alpha}{1-\alpha}.\end{eqnarray*}

But in our case we have:
$$\|R_V(z)\| = \left\|\sum_{k=0}^{d-1}\,\frac{\psi_{k,V}\psi_{k,V}^*}{z-\lambda_k(V)}\right\| = \max_{k}|z-\lambda_k(V)|^{-1} \leq 1/\gamma$$
because all $\lambda_k(V)$ lie within distance $\geq \gamma$ from
the contour $\sC$ (this follows from the assumption that no
$\lambda_k(V)$ is in $(a-\gamma,a+\gamma)\cup (b-\gamma,b+\gamma)$).
Moreover, $\|W-V\|\leq \eps$ by assumption. Therefore,
$\|(W-V)R_V(z)\|\leq \eps/\gamma<1$ and, by the above,
$$\|R_W(z) -R_V(z)\|\leq \frac{\eps}{\gamma^2 - \gamma\eps}.$$
Together with \eqnref{differenceresolvent}, this finishes the proof
for the finite-dimensional case.

We now consider the case of arbitrary $\sH$. Recall the definitions of $V_\delta$ and $W_\delta$ from the previous proof. It is easy to deduce
from the definition of $V_\delta$ that for any $v\in\C^d$,
$$\Pi_{a,b}(V)\,v = \lim_{\delta\searrow 0}\Pi_{a,b}(V_\delta)$$ and similarly
$$\Pi_{a,b}(W)\,v = \lim_{\delta\searrow 0}\Pi_{a,b}(W_\delta)\, v\mbox{
where }W_\delta\equiv \sum_{i:|\lambda_i|\geq
\delta}\eta_i\,\psi_{i,W}\psi_{i,W}^*.$$ Since $V_\delta$ and
$W_\delta$ have finite dimensional rank, one sees from the first
part that for all small enough $\delta>0$,
$$\|\left(\Pi_{a,b}(V_\delta)-\Pi_{a,b}(W_\delta)\right)\,v\|\leq
\|v\|\,\|\Pi_{a,b}(V_\delta)-\Pi_{a,b}(W_\delta)\|\leq \frac{(b-a
+\gamma)\,(\eps+2\delta)}{\pi(\gamma^2-\gamma(\eps+2\delta))}$$ since
$\|V_\delta - W_\delta\|\leq \eps+2\delta< \gamma$. Letting $\delta\searrow 0$ implies:
$$\|\left(\Pi_{a,b}(V)-\Pi_{a,b}(W)\right)\,v\|\leq
\|v\|\, \frac{(b-a +\gamma)\,\eps}{\pi(\gamma^2-\gamma\eps)}$$ and
since $v$ is arbitrary this finishes the proof.\end{proof}

\end{document}